\documentclass[12pt,reqno]{amsart}

\usepackage[top=28truemm,bottom=28truemm,left=26truemm,right=26truemm]{geometry}

\usepackage{amssymb, amsmath, amsthm}
\usepackage{amsfonts}
\usepackage{color}
\usepackage{mathrsfs}
\usepackage{graphicx}
\usepackage{comment}
\usepackage[abbrev]{amsrefs}
\usepackage{mathtools}
\usepackage{version}	
\numberwithin{equation}{section}
\renewcommand{\labelenumi}{(\roman{enumi})}

\renewcommand{\d}{\mathrm{d}}

\newcommand{\E}{\mathcal{E}}

\renewcommand{\H}{\mathcal{H}}

\newcommand{\N}{{\mathbb N}}

\newcommand{\R}{\mathbb R}

\newcommand{\loc}{\mathrm{loc}}

\newcommand{\dom}{{\partial\Omega}}
\newcommand{\ddom}{{\partial_D \Omega}}
\newcommand{\ndom}{{\partial_N \Omega}}

\newcommand{\ind}{{I_{(-\infty,0]}}}

\newcommand{\dind}{{\partial \ind}}

\newcommand{\ztau}{{z_\tau}}
\newcommand{\ztaubar}{\overline{z}_\tau}
\newcommand{\ftau}{{f_\tau}}
\newcommand{\ftaubar}{\overline{f}_\tau}

\newcommand{\vep}{\varepsilon}
\newcommand{\vphi}{\varphi}

\newcommand{\wto}{\rightharpoonup}

\newcommand{\disp}{\displaystyle}

\newcommand{\mel}[1]{\MoveEqLeft{#1}}

\theoremstyle{plain}
\newtheorem{thm}{Theorem}[section]
\newtheorem{prop}[thm]{Proposition}
\newtheorem{lem}[thm]{Lemma}

\theoremstyle{definition}
\newtheorem{defi}[thm]{Definition}

\theoremstyle{remark}
\newtheorem{rem}[thm]{Remark}

\title[Evolution with irreversibility and energy conservation]{Evolution equations with complete irreversibility\\and energy conservation}
\author{Goro Akagi}
\address{Mathematical Institute and Graduate School of Science, Tohoku University, Sendai 980-8578, Japan}
\email{goro.akagi@tohoku.ac.jp}

\author{Kotaro Sato}
\address{Mathematical Institute, Tohoku University, Sendai 980-8578, Japan}
\email{kotaro.sato.t8@dc.tohoku.ac.jp}

\date{May 8, 2023}

\thanks{G.A.~is partially supported by JSPS KAKENHI Grant Number JP21KK0044, JP21K18581, JP20H01812, JP18K18715 and JP20H00117. K.S.~is partially supported by JSPS KAKENHI Grant Number JP21J20732.}


\begin{document}

\begin{abstract}
This paper is concerned with the initial-boundary value problem for an evolutionary variational inequality complying with three intrinsic properties: \emph{complete irreversibility}, \emph{unilateral equilibrium of an energy} and \emph{an energy conservation law}, which cannot generally be realized in dissipative systems such as standard gradient flows. Main results consist of well-posedness in a strong formulation, qualitative properties of strong solutions (i.e., comparison principle and the three properties mentioned above) and long-time dynamics of strong solutions (more precisely, convergence to an equilibrium). The well-posedness will be proved based on a minimizing movement scheme without parabolic regularization, which will also play a crucial role for proving qualitative and asymptotic properties of strong solutions. Moreover, the variational inequality under consideration will be characterized as a singular limit of some (generalized) gradient flows.
\end{abstract}

\maketitle

\section{Introduction}\label{sec:intro}

Let $\Omega \subset \R^N$, $ N \geq 1 $ a bounded domain with Lipschitz boundary $\dom$ and let $T$ be a positive constant. Let $ \ddom $ and $ \ndom $ be relatively open subsets of $ \dom $ such that $ \ddom \cap \ndom = \emptyset $ and $ \H^{N-1}(\dom \setminus (\ddom \cup \ndom))=0 $, where $ \H^{N-1} $ denotes the $ (N-1) $-dimensional Hausdorff measure. The present paper is concerned with the following evolutionary variational inequality:
\begin{equation}\label{eq:vi}
\left\lbrace
\begin{alignedat}{4}
&\partial_t z \leq 0, \quad -\Delta z + \sigma z - f \leq 0 & \quad &\mbox{ in } \ \Omega \times (0,T),\\
&\partial_t z \left( -\Delta z + \sigma z - f \right) = 0 &&\mbox{ in } \ \Omega \times (0,T),
\end{alignedat}
\right.
\end{equation}
where $\sigma$ is a non-negative constant, equipped with the initial and boundary conditions,
\begin{equation}\label{bcic}
\left\lbrace
\begin{alignedat}{4}
&z = 0 & \quad &\mbox{ on } \ \ddom \times (0,T),\\
&\partial_\nu z = 0 &&\mbox{ on } \ \ndom \times (0,T),\\
&z(\cdot,0) = z_0 &&\mbox{ in } \ \Omega,
\end{alignedat}
\right.
\end{equation}
where $\partial_\nu$ stands for the outward normal derivative on $\partial_N \Omega$ and $z_0 = z_0(x)$ and $f = f(x,t)$ are prescribed data.
In a strong formulation (see Definition \ref{D:sol} below), the variational inequality \eqref{eq:vi} can be rewritten as a differential inclusion of the form,
\begin{equation}\label{eq}
\dind \left( \partial_t z \right) -\Delta z + \sigma z \ni f \quad \mbox{ in } \ \Omega \times (0,T),
\end{equation}
where $\ind : \R \to [0,+\infty]$ denotes the indicator function supported over $(-\infty, 0]$, i.e., $\ind(s) = 0$ if $s \leq 0$ and $\ind(s) = +\infty$ if $s > 0$, and $\dind : \R \to 2^\R$ stands for the subdifferential of $\ind$, that is,
\begin{align*}
\dind(s) &= \left\{ \xi \in \R \colon \xi(r-s) \leq 0 \quad \mbox{for } \, r \in \R \right\}\\
&= \begin{cases}
\{0\} &\mbox{ if } \ s < 0,\\
[0,+\infty) &\mbox{ if } \ s = 0,\\
\emptyset &\mbox{ if } \ s>0
\end{cases}
\quad \mbox{ for } \ s \in \R,
\end{align*}
with domain $D(\dind) = (-\infty,0]$.

On the other hand, set $H := L^2(\Omega)$ and 
\begin{equation}\label{V}
V := \{w \in H^1(\Omega) \colon \gamma(w) = 0 \quad \mbox{on } \, \partial_D \Omega\},
\end{equation}
where $\gamma : H^1(\Omega) \to H^{1/2}(\partial \Omega)$ stands for the trace operator (see also \S \ref{ssec:fsp} below). Define functionals $\psi,\varphi : H \to [0,+\infty]$ by
\begin{align}
\psi(w) &= I_K(w) \quad \mbox{ for } \ w \in H,\notag\\
\varphi(w) &= \begin{cases}\label{eq:EQ1}
\frac{1}{2} \int_\Omega |\nabla w(x)|^2 \, \d x + \frac \sigma 2 \int_\Omega \left| w(x) \right|^2 \d x \ &\mbox{ if } \ w \in V,\\
+\infty \ &\mbox{ otherwise}
\end{cases}
\quad \mbox{ for } \ w \in H,
\end{align}
where $I_K:H \to [0,+\infty]$ is the indicator function supported over the closed convex set $K := \{w \in H \colon w \leq 0 \ \mbox{ a.e.~in } \Omega\}$. Then \eqref{eq} equipped with \eqref{bcic} is reduced into the Cauchy problem for the following doubly-nonlinear evolution equation in $H$:
\begin{equation*}
\partial \psi (\partial_t z(t)) + \partial \varphi (z(t)) \ni f(t) \quad \mbox{ in } \ H, \quad 0 < t < T,
\end{equation*}
where $z(t):=z(\cdot,t)$, $f(t):=f(\cdot,t)$ and $\partial \psi, \partial \varphi : H \to 2^H$ stand for the subdifferential operators of $\psi$,~$\varphi$, respectively (see \S \ref{ssec:subdif} below). Such doubly-nonlinear evolution equations have been studied by many authors (see, e.g.,~\cite{Barbu75,Arai,Senba,CV,Colli,SSS,Roubicek,G15} and references therein). On the other hand, there seems no general theory which can cover the equation \eqref{eq} (or equivalently, \eqref{eq:vi}) (see also (i) of Remark \ref{rem:ex} below).

This study on \eqref{eq:vi}, \eqref{bcic} was originally inspired by variational models for brittle fracture, which will hence briefly be reviewed below. In order to extend the classical Griffith theory~\cite{Griffith} for brittle fracture, Francfort and Marigo~\cite{FM} introduced the so-called \emph{Francfort-Marigo energy} $\mathcal{F}(\vec{u},\Gamma;\vec{g})$, which is the sum of the elastic energy derived from the displacement field $\vec{u}$ and the surface energy of the crack set $\Gamma \subset \overline{\Omega}$ on an elastic body occupying $\Omega \subset \R^3$, and presented a quasistatic variational model as a minimization of $\mathcal{F}(\vec{u},\Gamma;\vec{g})$ in $(\vec{u}, \Gamma)$ subject to a (prescribed) boundary displacement field $ \vec{g} $ on $ \ddom $, a part of the boundary $\partial \Omega$, which is an interface between the elastic body and the source of the external forces (cf.~\eqref{bcic}) as well as an irreversibility for $\Gamma$. To be more precise, let us denote by $ \mathcal{F}(\Gamma;\vec{g}) $ the minimum value of $ \mathcal{F}(\vec{u},\Gamma;\vec{g}) $ with respect to $ \vec{u} : \Omega \setminus \Gamma \to \R^3 $ satisfying $ \vec{u} = \vec{g} $ on $ \ddom \setminus \Gamma $. Here the crack set $ \Gamma $ is excluded from the Dirichlet boundary for $\vec{u}$, since the boundary displacement is supposed to be not transmitted through the fractured part of the boundary. In~\cite{FM}, for the boundary displacement $\vec{g} = \vec{g}(x,t)$ varying in time, a \emph{quasi-static evolution of crack} $ t \mapsto \Gamma (t) $ on $[0,T]$ is defined to satisfy the following three conditions:
\begin{enumerate}
\item $ \Gamma(t_1) \subset \Gamma(t_2) \, $ for all $\, 0 \leq t_1 \leq t_2 \leq T$;
\item for each $ t \in [0,T]$, it holds that $ \mathcal{F}(\Gamma(t);\vec{g}(\cdot,t)) \leq \mathcal{F}(\Gamma;\vec{g}(\cdot,t)) $ for any closed subsets $ \Gamma \subset \overline{\Omega}$ satisfying $ \bigcup_{s<t} \Gamma(s) \subset \Gamma $;
\item the total energy function $ t \mapsto \mathcal{F}(\Gamma(t);\vec{g}(\cdot,t)) $ is absolutely continuous on $[0,T]$ and its derivative is balanced with a temporal variation of external forces.
\end{enumerate}
Condition (i) constrains the crack evolution to be completely \textit{irreversible}; (ii) can be regarded as a \emph{unilateral minimization of the energy}, that is, at each time, the crack set in the elastic body minimizes the Francfort-Marigo energy among virtual crack sets larger than the previous ones. It also describes a quasi-static evolution of the crack set; (iii) is none other than the \textit{energy conservation} of the whole system. In particular, crack never evolves unless the external load $\vec{g}(x,t)$ varies in time. All these conditions are intrinsic nature to crack evolution. For this variational fracture model, existence results on the quasi-static evolution $ t \mapsto \Gamma(t) $ have already been obtained, e.g., in~\cite{DT}, where anti-planar shear fracture in two-dimensional cylindrical elastic bodies is considered, and moreover, in~\cite{DFT}, where the result in~\cite{DT} is extended to more general configurations such as nonlinear elasticity for heterogeneous and anisotropic materials in general dimension. We further refer the reader to, e.g.,~\cite{Cha03,FL03} with no claim of completeness.

Moreover, in~\cite{BFM00} (see also~\cite{BFM08}), in view of numerical analysis (cf.~see also~\cite{Ng03,Ng05} for another approach), the Francfort-Marigo energy is regularized based on the idea of Ambrosio and Tortorelli (see~\cite{AT1990,AT1992}) by introducing a phase-field $z = z(x,t) \in [0,1]$ which takes $1$ (respectively, $0$) on the completely sound (respectively, cracked) part of the material. Such a regularized functional $\mathcal{F}_\vep(\vec{u},z;\vec{g})$ with an approximation parameter $\vep > 0$ is often called an \emph{Ambrosio-Tortorelli functional}. Moreover, in~\cite{Giacomini}, Giacomini constructed a quasi-static evolution $t \mapsto (\vec{u}_\vep(t), z_\vep(t))$ of crack for the Ambrosio-Tortorelli functional $\mathcal{F}_\vep$ for each $\vep > 0$ and further proved that it converges to a quasi-static evolution $t \mapsto (\vec{u}(t),\Gamma(t))$ for the Francfort-Marigo energy $\mathcal{F}$ as $\vep \to 0_+$ by the use of the $\Gamma$-convergence $\mathcal{F}_\vep \to \mathcal{F}$ (see~\cite{AT1990,AT1992}). 
Besides, the quasi-static evolution minimizing the Ambrosio-Tortorelli functional $\mathcal{F}_\vep$ still inherits the three properties (i)--(iii) mentioned above for the Francfort-Marigo energy, while the monotone growth of the crack set $\Gamma(t)$ as in (i) is replaced by the monotonicity in time of the phase-field $z = z(x,t)$. 

In~\cite{Giacomini}, the construction of the quasi-static evolution $t \mapsto (\vec{u}_\vep(t), z_\vep(t))$ for the Ambrosio-Tortorelli functional is based on the so-called \emph{minimizing movement scheme}. For simplicity, let us restrict ourselves to the following Ambrosio-Tortorelli functional for a scalar displacement field $u$ and a phase-field $ z $:
\[
\mathcal{F}_\vep(u,z;g) = \frac 12 \int_\Omega (\eta_\vep + z^2) \left| \nabla u \right|^2 \d x + \frac \vep 2 \int_\Omega |\nabla z|^2 \, \d x + \frac 1 {2\vep} \int_\Omega (1-z)^2 \, \d x,
\]
where $(u,z) \in H^1(\Omega) \times H^1(\Omega)$ satisfies $ z = 1 $, $u = g$ on $ \ddom $, a part of $\partial \Omega$, and $0 \leq z \leq 1$ a.e.~in $\Omega$, and $0 < \eta_\vep \ll \vep$ is a constant. Let $m \in \N$, $\tau := T/m > 0$ and $t^j := j \tau$ for $j=0,1,\dots,m$. Given $(u^{j-1},z^{j-1}) \in H^1(\Omega) \times H^1(\Omega)$, we first find a minimizer $(u^j,z^j) \in H^1(\Omega) \times H^1(\Omega)$ of $\mathcal F_\vep(u,z;g^j)$, where $g^j $ approximates $ g(t^j)$, subject to the constraint $z \leq z^{j-1}$ a.e.~in $\Omega$ (which corresponds to the irreversibility of crack evolution) and $ z^j = 1 $, $u^j = g^j$ on $ \ddom $. Then a quasi-static evolution $t \mapsto (u_\vep(t),z_\vep(t)) \in H^1(\Omega) \times H^1(\Omega)$ for the Ambrosio-Tortorelli functional $\mathcal{F}_\vep$ can be obtained as a limit of the piecewise constant interpolants of $\{u^j,z^j\}_{j=1,\ldots,m}$ as $m \to +\infty$ (equivalently, $\tau \to 0_+$).

On the other hand, the quasi-static evolution has not yet been fully pursued in view of the Euler-Lagrange equation. The discrete pair $(u^j,z^j)$ may be characterized as a (weak) solution to the following discretized Euler-Lagrange equation:
\begin{alignat*}{4}
-\mathrm{div} \left( (\eta_\vep + (z^j)^2) \nabla u^j \right) &= 0 & \quad &\mbox{ in } \ \Omega,\\
\dind \left( z^j-z^{j-1} \right) - \vep \Delta z^j + \frac 1\vep (z^j-1) + z^j \big| \nabla u^j \big|^2 &\ni 0 &&\mbox{ in } \ \Omega,
\end{alignat*}
whose formal continuous limit reads,
\begin{alignat*}{4}
-\mathrm{div} \left( (\eta_\vep + z^2) \nabla u \right) &= 0 & \quad & \mbox{ in } \ \Omega \times (0,T),\\
\dind \left( \partial_t z \right) - \vep \Delta z + \frac1{\vep} \left( z-1 \right) + z |\nabla u|^2 &\ni 0 &&\mbox{ in } \ \Omega \times (0,T).
\end{alignat*}
In particular, there still seems no result which can guarantee the well-posedness for the system of PDEs mentioned above. Equation \eqref{eq} is derived as a simplified form of the second equation (i.e., a phase-field equation) of the system. It is noteworthy that \eqref{eq} also complies with three properties (i.e., complete irreversibility, unilateral equilibrium of an energy and an energy conservation law) similar to those for the quasi-static evolution of brittle fracture (see Theorem \ref{thm:qua} below for details). In~\cite{AK}, the following equation was comprehensively studied:
$$
\partial_t z + \dind \left( \partial_t z \right) - \Delta z \ni f \quad \mbox{ in } \ \Omega \times (0,T),
$$
which can be seen as a parabolic regularization for \eqref{eq} with $\sigma = 0$ (see~\cite{K-S80,GiSa94,GiGoSa94} for $f \equiv 0$ and also~\cite{KiNe} for a weak formulation). Actually, in~\cite{AK}, some parabolic smoothing effect and energy-dissipation were also observed. On the other hand, due to the parabolic regularization, neither the unilateral equilibrium nor the conservation of an energy hold true. Moreover, such a parabolic regularization has also been applied to phase-field models for brittle fractures (see, e.g.,~\cite{FN0,FN,FKNS,Fr1,Fr2,Ned1,BS,BSS,MR,Mi,KRZ,Ned2}); both the unilateral equilibrium of the energy and the energy conservation law can no longer be realized there. The present paper is concerned with an evolution equation complying with the three properties, that is, complete irreversibility of evolution, unilateral equilibrium of an energy and an energy conservation law. We aim at proving the well-posedness for the initial-boundary value problem and investigating qualitative and asymptotic properties of its solutions. The main results of the present paper may not yet been applicable directly to the fracture mechanics; however, they shed a light to give a better understanding of the quasi-static evolution of fractures developed in literature (e.g.,~\cite{FM,DT,DFT,Cha03,FL03,BFM00,Giacomini}) from a more PDE-based point of view (via the phase-field approximation) with keeping the three intrinsic properties to fractures models.

In order to state main results, in addition to \eqref{V}, we define
\begin{align*}
X &:= \left\{ u \in H^2(\Omega) \colon \partial_\nu u=0 \quad \mbox{on } \, \ndom \right\}
\end{align*}
(see also \S \ref{ssec:fsp} below). Moreover, let $V^*$ and $X^*$ denote the dual spaces of $V$ and $X$, respectively. In what follows, we are concerned with \emph{strong solutions} to \eqref{eq:vi}, \eqref{bcic} defined as follows:

\begin{defi}[Strong solution for \eqref{eq:vi}, \eqref{bcic}]\label{D:sol}
A function $z \in C([0,T];L^2(\Omega))$ is called a \emph{strong solution} on $[0,T]$ to \eqref{eq:vi} (or equivalently, \eqref{eq}) and \eqref{bcic}, if the following (i)--(iv) are all satisfied\/{\rm :}
\begin{enumerate}
\item[(i)] $z \in W^{1,2}(0,T;L^2(\Omega)) \cap L^2(0,T;X \cap V)$;
\item[(ii)] $\partial_t z \leq 0 \ $ and $ \ -\Delta z + \sigma z - f \leq 0$\quad a.e.~in $ \ \Omega \times (0,T)$;
\item[(iii)] $\partial_t z (-\Delta z + \sigma z - f) = 0$\quad a.e.~in $\ \Omega \times (0,T)$;
\item[(iv)] $z(x,0) = z_0(x)$\quad for a.e.~$x \in \Omega$.
\end{enumerate}
\end{defi}

Moreover, we define an operator $B:V \to V^*$ by
\begin{equation}\label{eq:B}
\left\langle Bu, v \right\rangle_V := \int_\Omega \left( \nabla u \cdot \nabla v + uv \right) \d x \quad \mbox{ for } \ u, v \in V.
\end{equation}
Then $B$ is bijective from $V$ to $V^*$, and hence, the inverse mapping $B^{-1} : V^* \to V$ is well defined. Throughout this paper, we assume that
\begin{equation}\label{eq:ass}
B^{-1} w \in H^2(\Omega) \cap V \quad \mbox{ for } \ w \in L^2(\Omega),
\end{equation}
which is nothing but the $L^2$-\emph{elliptic regularity} (i.e., $u \in H^2(\Omega)$ if $-\Delta u + u\in L^2(\Omega)$ along with an appropriate boundary condition; see, e.g.,~\cite[\S 9.6]{HB}) and always holds true for smooth domains along with either Dirichlet or Neumann boundary condition; on the other hand, it is rather delicate for mixed boundary conditions as well as for non-smooth domains. Indeed, according to~\cite{LM}, there are some elliptic operators equipped with certain mixed boundary conditions for which the elliptic regularity condition \eqref{eq:ass} does not hold true (see also~\cite{Grisvard, MS} and the references therein). The mixed boundary condition \eqref{bcic} is motivated from the fracture model mentioned above (see, e.g.,~\cite{Giacomini}); therefore, we assumed \eqref{eq:ass} to guarantee the $H^2$-regularity for the domain of the Laplacian equipped with \eqref{bcic}.

Main results of the present paper will be stated below. First of all, we exhibit that
\begin{thm}[Well-posedness]\label{thm:main}
Assume that \eqref{eq:ass} is satisfied. Let $z_0 \in X \cap V$, $f \in L^2 (0,T;L^2(\Omega)) \cap W^{1,2} (0,T;V^*)$ and let $\sigma \geq 0$ be a constant. Assume that
\renewcommand{\labelenumi}{(\roman{enumi})}
\begin{enumerate}
\item $\sigma > 0 \ $ if $\ \H^{N-1}(\ddom)=0${\rm ;}
\item there exists $\hat{f} \in L^2(\Omega) \,$ such that $\, f(x,t) \geq \hat{f}(x) \,$ for a.e.~$(x,t) \in \Omega \times (0,T)${\rm ;}
\item $-\Delta z_0 + \sigma z_0 - f(\cdot,0) \leq 0 \,$ in $\, V^*$.
\end{enumerate}
Then the initial-boundary value problem \eqref{eq:vi} {\rm(}or equivalently, \eqref{eq}{\rm )} and \eqref{bcic} admits a unique strong solution $z = z(x,t)$ such that
\begin{equation}\label{eq:A}
z \in W^{1,2} (0,T;V).
\end{equation}
Moreover, for each $i=1,2$, let $z^i_0 = z^i_0(x)$ and $f^i = f^i(x,t)$ be given functions satisfying all the assumptions mentioned above and let $z^i = z^i(x,t)$ denote the strong solution to \eqref{eq:vi} {\rm(}equivalently, \eqref{eq}{\rm )} and \eqref{bcic} with $z_0=z^i_0$ and $f=f^i$. Then there exists a constant $C>0$ independent of $T$ such that
\begin{align}
\MoveEqLeft{
\sup_{t \in [0,T]} \left\| z^1(t)-z^2(t) \right\|_V
}\nonumber\\
&\leq C \, \Big( \left\| z^1_0-z^2_0 \right\|_V + \sup_{t \in [0,T]} \left\| f^1(t)-f^2(t) \right\|_{V^*} + \left\| \partial_t f^1 - \partial_t f^2 \right\|_{L^1(0,T;V^*)} \Big).\label{conti-dep}
\end{align}
\end{thm}

The next theorem is concerned with the three intrinsic properties of \eqref{eq:vi} (equivalently, \eqref{eq}). Here we define an energy functional $\mathcal{E} : V \times [0,T]  \to \R$ by
\begin{align}
\mathcal E(w,t) &= \varphi(w) - \left\langle f(t),w \right\rangle_V \nonumber\\
&= \frac{1}{2} \int_{\Omega} |\nabla w(x)| \, \d x + \frac{\sigma}{2} \int_{\Omega} |w(x)|^2 \, \d x - \left\langle f(t), w \right\rangle_V \quad \mbox{ for } \ w \in V \ \mbox{ and } \ t \in [0,T]. \label{eq:e}
\end{align}

\begin{thm}[Irreversibility, unilateral equilibrium and conservation of the energy]\label{thm:qua}
Assume that \eqref{eq:ass} is satisfied. Let $\sigma \geq 0$, $f \in L^2(0,T;L^2(\Omega)) \cap W^{1,2}(0,T;V^*)$ and $z_0 \in X \cap V$ satisfy all the assumptions {\rm (i)--(iii)} of Theorem {\rm \ref{thm:main}}. Then the unique strong solution $z = z(x,t)$ to \eqref{eq:vi} {\rm(}or equivalently, \eqref{eq}{\rm )} and \eqref{bcic} enjoys the following three conditions\/{\rm :}
\renewcommand{\labelenumi}{(\roman{enumi})}
\begin{enumerate}
\item {\rm (}{\bf Complete irreversibility of evolution}{\rm )} It holds that
\[
\partial_t z \leq 0 \quad \mbox{ a.e.~in } \ \Omega \times (0,T);
\]
\item {\rm (}{\bf Unilateral equilibrium of the energy}{\rm )} for all $t \in [0,T]${\rm ,} it holds that
\[
\E ( z(t),t ) \leq \E ( v,t )
\]
for any $v \in V$ satisfying $v \leq z(t)$ a.e.~in $\, \Omega${\rm ;}
\item {\rm (}{\bf Energy conservation law}{\rm )} the function $t \mapsto \E(z(t),t)$ is absolutely continuous on $[0,T]$ and the energy balance holds in the following sense\/{\rm :}
\[
\E ( z(t),t ) - \E ( z(s),s ) = -\int^t_s \left\langle \partial_t f(r), z(r) \right\rangle_V \d r
\]
for all $t, s \in [0,T]$. In particular, if $f$ is stationary in time, that is, $\partial_t f \equiv 0$, then no evolution occurs, i.e., $z(t) \equiv z_0$.
\end{enumerate}
\end{thm}

Now, we move on to global dynamics of \emph{global-in-time strong solutions}, which are defined below, for the initial-boundary value problem \eqref{eq:vi} (or equivalently, \eqref{eq}) and \eqref{bcic}.

\begin{defi}[Global-in-time strong solution]
A function $z \in C([0,\infty);L^2(\Omega))$ is called a \emph{global-in-time strong solution} to the initial-boundary value problem \eqref{eq:vi} (or equivalently, \eqref{eq}) and \eqref{bcic}, if the following (i) and (ii) are satisfied\/{\rm :}
\begin{enumerate}
\item $z \in W^{1,2}_{\rm loc} ([0,\infty);L^2(\Omega)) \cap L^2_{\rm loc} ([0,\infty);X \cap V)$;
\item for each $T>0$, the restriction $z|_{[0,T]}$ of $ z $ onto $ [0,T] $ is a strong solution on $[0,T]$ to \eqref{eq:vi} and \eqref{bcic}.
\end{enumerate}
\end{defi}

Then we have

\begin{thm}[Global well-posedness]\label{thm:global}
Assume that \eqref{eq:ass} is satisfied. Let $\sigma \geq 0$, $f \in W^{1,2}_{\rm loc}([0,\infty);V^*) \cap L^2_{\rm loc}([0,\infty);L^2(\Omega))$ and $z_0 \in X \cap V$ satisfy all the assumptions {\rm (i)--(iii)} of Theorem {\rm \ref{thm:main}} for any $T > 0$. Then there exists a global-in-time strong solution $z \in C([0,\infty);L^2(\Omega))$ to \eqref{eq:vi} {\rm (}or equivalently \eqref{eq}{\rm )} and \eqref{bcic} such that
\[
z \in W^{1,2}_\loc ([0,\infty);V).
\]

In addition, if $\partial_t f \in L^2(0,\infty;V^*)$, then it holds that $\partial_t z \in L^2(0,\infty;V)$ and 
\begin{equation}\label{dtz-bdd}
\left\| \partial_t z \right\|_{L^2(0,\infty;V)} \leq C \left\| \partial_t f \right\|_{L^2(0,\infty;V^*)}
\end{equation}
for some constant $C > 0$.

Moreover, let $f^j = f^j(x,t)$ and $z_0^j = z_0^j(x)$ satisfy all the assumptions above and let $z^j = z^j(x,t)$ be the global-in-time strong solutions to \eqref{eq:vi} {\rm (}or equivalently \eqref{eq}{\rm )} and \eqref{bcic} with $f = f^j$ and $z_0 = z_0^j$ for $j=1,2$. Then there exists a constant $C > 0$ such that \eqref{conti-dep} holds for each $T > 0$.
\end{thm}

The following theorem then concerns the long-time behavior of global-in-time strong solutions for \eqref{eq:vi}, \eqref{bcic}.

\begin{thm}[Long-time dynamics of global-in-time strong solutions]\label{thm:asympt}
Assume that \eqref{eq:ass} is satisfied. Let $\sigma \geq 0$, $f \in W^{1,2}_{\rm loc}([0,\infty);V^*) \cap L^2_{\rm loc}([0,\infty);L^2(\Omega))$ and $z_0 \in X \cap V$ satisfy the assumptions {\rm (i)} and {\rm (iii)} of Theorem {\rm \ref{thm:main}} and
\begin{enumerate}
\item[$(\mathrm{ii})_{T=\infty}$] there exists $\hat{f} \in L^2(\Omega) $ such that $ f(x,t) \geq \hat{f}(x) $ for a.e.~$(x,t) \in \Omega \times (0,\infty)$.
\end{enumerate}
In addition, assume that
\begin{enumerate}
\setcounter{enumi}{3}
\item $f \in L^\infty(0,\infty;L^2(\Omega))$, $\partial_t f \in L^2(0,\infty;V^*)${\rm ;}
\item there exists a function $f_\infty \in L^2(\Omega)$ such that $f-f_\infty \in L^2(0,\infty;V^*)$.
\end{enumerate}

Let $z = z(x,t)$ be the global-in-time strong solution to \eqref{eq:vi} {\rm (}or equivalently \eqref{eq}{\rm )} and \eqref{bcic}. Then there exists a function $z_\infty \in X \cap V$ such that
\[
z(t) \to z_\infty \quad \mbox{ strongly in } \ V \quad \mbox{ as } \ t \to \infty
\]
and
\[
z_\infty \leq z_0, \quad -\Delta z_\infty + \sigma z_\infty - f_\infty \leq 0 \quad \mbox{ a.e.~in } \ \Omega.
\]
In addition, if it is also satisfied that
\begin{equation}\label{eq:xi}
f_\infty \leq f \quad \mbox{ a.e.~in } \ \Omega \times (0,\infty),
\end{equation}
then the stationary limit $z_\infty$ is the unique solution to the variational inequality,
\begin{equation*}
\left\lbrace
\begin{alignedat}{4}
&z_\infty \leq z_0, \quad -\Delta z_\infty + \sigma z_\infty - f_\infty \leq 0 & \quad &\mbox{ in } \ \Omega,\\
&(z_\infty-z_0) \left( -\Delta z_\infty + \sigma z_\infty - f_\infty \right) = 0 &&\mbox{ in } \ \Omega,
\end{alignedat}
\right.
\end{equation*}
which is equivalently rewritten as the inclusion,
\begin{equation}\label{eq:viinfty}
\dind \left( z_\infty - z_0 \right) - \Delta z_\infty + \sigma z_\infty \ni f_\infty \ \mbox{ in } \ \Omega,
\end{equation}
equipped with the boundary condition,
\begin{equation}\label{bc}
\left\lbrace
\begin{alignedat}{4}
&z_\infty = 0 & \quad &\mbox{ on } \ \ddom,\\
&\partial_\nu z_\infty = 0 &&\mbox{ on } \ \ndom.
\end{alignedat}
\right.
\end{equation}
\end{thm}

We shall finally characterize \eqref{eq:vi} (or equivalently \eqref{eq}) as a singular limit of the parabolic equation,
\begin{equation}\label{eq-ep}
\vep \partial_t z_\vep + \dind \left( \partial_t z_\vep \right) - \Delta z_\vep + \sigma z_\vep \ni f_\vep \quad \mbox{ in } \ \Omega \times (0,T),
\end{equation}
equipped with the initial and boundary conditions,
\begin{equation}\label{bcic-ep}
\left\lbrace
\begin{alignedat}{4}
&z_\vep = 0 & \quad &\mbox{ on } \ \ddom \times (0,T),\\
&\partial_\nu z_\vep = 0 &&\mbox{ on } \ \ndom \times (0,T),\\
&z_\vep(\cdot,0) = z_{0,\vep} &&\mbox{ in } \ \Omega.
\end{alignedat}
\right.
\end{equation}
More precisely, we shall prove convergence of strong solutions $z_\vep = z_\vep(x,t)$ for \eqref{eq-ep}, \eqref{bcic-ep} to the strong solution $z = z(x,t)$ for \eqref{eq:vi}, \eqref{bcic} as $\vep \to 0_+$. Our result reads,

\begin{thm}[Characterization as a singular limit]\label{thm:conv}
Assume that \eqref{eq:ass} is satisfied.
For each $\vep > 0$, let $f_\vep \in L^2(0,T;L^2(\Omega)) \cap W^{1,2}(0,T;V^*)$ satisfy {\rm (ii)} of Theorem {\rm \ref{thm:main}} with $f$ replaced by $f_\vep$ and let $z_{0,\vep} \in V$ satisfy
\begin{equation*}
\Delta z_{0,\vep} \in \mathcal M(\overline{\Omega}) \quad \mbox{ and } \quad (\Delta z_{0,\vep})_- \in L^2(\Omega),
\end{equation*} 
where $\mathcal{M}(\overline{\Omega})$ denotes the set of signed Radon measures and $ \mu_- $ stands for the negative part of $ \mu \in \mathcal{M}(\overline{\Omega}) $ defined by the Hahn-Jordan decomposition.
Then there exists a unique strong solution $ z_\vep \in W^{1,2}(0,T;L^2(\Omega)) \cap L^2(0,T;X \cap V) $ to \eqref{eq-ep}, \eqref{bcic-ep}. Moreover, let $ z_0 \in X \cap V $ and $ f \in L^2(0,T;L^2(\Omega)) \cap W^{1,2}(0,T;V^*) $ satisfy the assumptions {\rm (i)--(iii)} of Theorem {\rm \ref{thm:main}} and let $ z \in W^{1,2}(0,T;V) \cap L^2(0,T;X) $ be the strong solution to \eqref{eq:vi} {\rm (}or equivalently, \eqref{eq}{\rm )} and \eqref{bcic}. Then there exists a constant $C > 0$ independent of $ \vep $ and $T$ such that
\begin{align*}
\sup_{t \in[0,T]} \left\| z_\vep(t)-z(t) \right\|_V &\leq C \, \Big( \left\| z_{0,\vep}-z_0 \right\|_V + \sup_{t \in [0,T]} \left\| f_\vep(t)-f(t) \right\|_{V^*}\\
&\quad + \sqrt{\vep} \left\| \partial_t z \right\|_{L^2(0,T;L^2(\Omega))} + \left\| \partial_t f_\vep - \partial_t f \right\|_{L^1(0,T;V^*)} \Big).
\end{align*}
In particular, if $z_{0,\vep} \to z_0$ strongly in $V$ and $f_\vep \to f$ strongly in $W^{1,1}(0,T;V^*)$ as $\vep \to 0_+$, then
\[
z_\vep \to z \quad \mbox{ strongly in } \ C \left( [0,T];V \right).
\]
\end{thm} 

\bigskip
\noindent
{\bf Plan of the paper.} \ The present paper is composed of six sections. In Section \ref{sec:pre}, we shall set up notation and recall some preliminary facts which will be used in later sections. Section \ref{sec:main} is devoted to proving Theorem \ref{thm:main} via the so-called minimizing movement scheme without parabolic regularization. In Section \ref{sec:qua}, we shall first prove comparison principle for \eqref{eq:vi} (or \eqref{eq}) (see \S \ref{ssec:comp} below), and then we shall verify Theorem \ref{thm:qua} (see \S \ref{ssec:qua} below). Section \ref{sec:global} concerns global dynamics of global-in-time solutions for the initial-boundary value problem. We shall first prove Theorem \ref{thm:global} and then verify Theorem \ref{thm:asympt} concerning convergence of each global-in-time solution to an equilibrium, which can be characterized as a solution to some variational inequality of obstacle type. The final section is devoted to discussing a singular limit of strong solutions for \eqref{eq-ep}, \eqref{bcic-ep} as $\vep \to 0_+$ to characterize strong solutions for \eqref{eq:vi} (or equivalently \eqref{eq}), \eqref{bcic} (see Theorem \ref{thm:conv}).

\section{Notation and preliminaries}\label{sec:pre}

In this section, we set up notation and recall some preliminary facts for later use.

\subsection{Notation}\label{ssec:note}

For $a,b \in \R$, we denote by $a \wedge b$ and $a \vee b$ the minimum and maximum of $a$ and $b$, respectively. This notation will also be used for functions $f, g : M \to \R$ from a set $M$ to $\R$, i.e., we write
\[
\left( f \wedge g \right) (x) := f(x) \wedge g(x) \quad \mbox{ and } \quad (f \vee g) (x) := f(x) \vee g(x) \quad \mbox{ for } \ x \in M.
\]
For $d \in \N = \mathbb{Z}_{\geq 1}$, let $\H^d ( \cdot )$ denote the $d$-dimensional Hausdorff measure. Let $\Omega$ be an open subset of $\R^N$. For $1 \leq p \leq +\infty$ and $m \in \N$, we denote by $L^p(\Omega)$ and $W^{m,p}(\Omega)$ the Lebesgue and Sobolev spaces, respectively, and write $H^m(\Omega) = W^{m,2}(\Omega)$. For Lebesgue measurable sets $B$ in $\R^d$, $ d \in \N $, we denote by $|B|$ the $d$-dimensional Lebesgue measure of $B$.

For a Banach space $E$, let $E^*$ denote its dual space. The duality pairing of $u \in E$ and $f \in E^*$ is denoted by $\left\langle f,u \right\rangle_E$. Let $I$ be a (non-empty) open interval on $\R$. For $1 \leq p \leq +\infty$, we denote by $L^p(I;E)$ and $W^{1,p}(I;E)$ the Bochner space and Sobolev-Bochner space, respectively. Moreover, we often write $\partial_t = \partial / \partial t$.

Furthermore, we denote by $C$ a non-negative constant which is independent of the elements of the corresponding space or set and may vary from line to line.

\subsection{Subdifferential operator}\label{ssec:subdif}

Let $H$ be a real Hilbert space equipped with an inner product $(\cdot,\cdot)_H$ and let $\vphi : H \to (-\infty, +\infty]$ be a proper, i.e., $\varphi \not\equiv +\infty$, lower-semicontinuous convex functional, whose \emph{effective domain} is defined by $D(\vphi) = \{ u \in H \colon \vphi(u) < +\infty \}$. The \emph{subdifferential operator} $\partial \vphi : H \to 2^H$ of $\varphi$ is defined by
$$
\partial \vphi(u) = \left\{ \xi \in H \colon \vphi(v)-\vphi(u) \geq ( \xi, v-u )_H \quad \mbox{for } \ v \in D(\vphi) \right\} \ \mbox{ for } \ u \in H
$$
(here we set $\partial \varphi(u) = \emptyset$ when $u \not\in D(\varphi)$) with domain $D(\partial \vphi) := \{ u \in H \colon \partial \vphi(u) \neq \emptyset \}$. It is well known that $ \partial \vphi $ is maximal monotone in $H \times H$ (see, e.g.,~\cite{HBF} for more details).

\subsection{Sobolev spaces for mixed boundary conditions}\label{ssec:fsp}

Let $\Omega$ be a bounded domain of $\R^N$ with Lipschitz boundary $\dom$.
Let $ \ddom $ and $ \ndom $ be relatively open subsets of $ \dom $ such that $ \H^{N-1} (\dom \setminus (\ddom \cup \ndom)) = 0 $ and $ \ddom \cap \ndom = \emptyset $.
We set
\begin{align*}
V &= \{ u \in H^1(\Omega) \colon \gamma(u)=0 \quad \H^{N-1} \mbox{-a.e.~on } \, \ddom \}, \\
X &= \{ u \in H^2(\Omega) \colon \partial_\nu u =0 \quad \H^{N-1} \mbox{-a.e.~on } \, \ndom \},
\end{align*}
where $\gamma$ stands for the trace operator, $\nu$ denotes the outer unit normal vector field on $\dom$ and $\partial_\nu$ denotes the normal derivative on $\dom$, i.e., $\partial_\nu u = \gamma(\nabla u) \cdot \nu$ for $ u \in H^2(\Omega) $. We shall omit the trace operator when no confusion can arise. We set 
\[
\| v \|_V := \| v \|_{H^1(\Omega)} \quad \mbox{ for } \ v \in V \quad \mbox{ and } \quad \| v \|_X := \| v \|_{H^2(\Omega)} \quad \mbox{ for } \ v \in X.
\]
Moreover, we shall write $f \leq g$ in $V^*$ for $f,g \in V^*$, if $\langle f, v \rangle_V \leq \langle g, v \rangle_V$ for any $v \in V$ satisfying $v \geq 0$ a.e.~in $\Omega$.

\subsection{A chain-rule formula}\label{ssec:chain}

Define $\phi : L^2(\Omega) \to [0,+\infty]$ by
\begin{equation}\label{eq:phi}
\phi (v) := \begin{cases}
\frac{1}{2} \int_\Omega \left| \nabla v(x) \right|^2 \d x &\mbox{ if } \ v \in V,\\
+\infty &\mbox{ if } \ v \in L^2(\Omega) \setminus V.
\end{cases}
\end{equation}
Then the following fact may be well known (see, e.g.,~\cite{HB,HB3,Barbu},~\cite[Lemmas 3.1 and 3.4]{AK}):

\begin{lem}\label{lem:phi}
The functional $\phi$ defined by \eqref{eq:phi} is convex and lower-semicontinuous in $L^2(\Omega)$. Moreover, if \eqref{eq:ass} is satisfied, it then holds that
\begin{equation*}
D(\partial \phi) = X \cap V \quad \mbox{ and } \quad \partial \phi(v) = -\Delta v \quad \mbox{ for } \ v \in X \cap V.
\end{equation*}
Furthermore, let $u \in W^{1,2}(0,T;L^2(\Omega)) \cap L^2(0,T;X \cap V)$. Then $u \in C([0,T];V)$, and moreover, the function
\[
t \mapsto \phi (u(t)) = \frac{1}{2} \int_\Omega \left| \nabla u(x,t) \right|^2 \d x
\]
is absolutely continuous on $[0,T]$ and complies with the chain-rule formula,
\begin{equation}\label{eq:chain}
\frac{\d}{\d t} \, \phi ( u(t) ) = \left( \partial_t u(t), - \Delta u(t) \right)_{L^2(\Omega)} \quad \mbox{ for a.e.~} t \in (0,T).
\end{equation}
\end{lem}

\section{Well-posedness}\label{sec:main}

This section is devoted to proving Theorem \ref{thm:main}. To this end, we shall prove well-posedness for \eqref{eq}, which is equivalent to \eqref{eq:vi} in the strong formulation (see Definition \ref{D:sol}). Before proceeding to a proof, we give the following

\begin{rem}\label{rem:ex}
\begin{enumerate}
\item Whenever $f$ belongs to $W^{1,1}(0,T;L^2(\Omega))$, assumption (ii) of Theorem \ref{thm:main} follows with $\hat{f} \in L^2(\Omega)$ given by
\[
\hat{f}(x) := f(x,0) - \int^T_0 \left| \partial_t f (x,s) \right| \d s \quad \mbox{ for } \ x \in \Omega.
\]
Indeed, we see that
\[
f(x,t) = f(x,0) + \int^t_0 \partial_t f(x,s) \,\d s \geq \hat{f}(x) \quad \mbox{ for a.e.~} (x,t) \in \Omega \times (0,T)
\]
by virtue of the fundamental theorem of calculus. In~\cite[Example 1, p.119]{HB2}, existence and uniqueness of strong solutions for \eqref{eq:vi} with $\sigma = 0$ and the homogeneous Dirichlet condition are also proved for $f \in W^{1,1}(0,T;L^2(\Omega))$. On the other hand, the assumptions on $f$ in Theorem \ref{thm:main} is strictly weaker than $f \in W^{1.1}(0,T;L^2(\Omega))$. Indeed, for instance, set $f(x,t) = - |x-t|^{1-\alpha}$ with $\alpha \in (1/2,1)$ for $x,t \in (0,1)$. Then one may easily check $f \geq \hat f \equiv -1$ and $f \in L^\infty((0,1) \times (0,1)) \subset L^2(0,1;L^2(0,1))$. Moreover, noting that $\partial_t f = \partial_x f$, we find that $\|\partial_t f(\cdot,t)\|_{H^{-1}(0,1)} \leq \|f(\cdot,t)\|_{L^2(0,1)}$ for $t \in (0,1)$, and hence, $\partial_t f \in L^2(0,1;H^{-1}(0,1))$. However, by virtue of $\alpha > 1/2$, $\partial_t f(\cdot,t)$ never lies on $L^2(0,1)$ for any $t \in (0,1)$. Therefore $f$ complies with all the assumptions as in Theorem \ref{thm:main} with $\Omega = (0,1)$, $\ndom = \emptyset$ and $T = 1$, but it never belongs to $W^{1,1}(0,1;L^2(0,1))$.
\item The assumption (iii) of Theorem \ref{thm:main} can be regarded as an admissible condition for the initial data, and it is also necessary for existence of strong solutions to \eqref{eq:vi}, \eqref{bcic}. Indeed, if there exists a strong solution $z = z(x,t)$ to \eqref{eq:vi} and \eqref{bcic}, then we see that
\begin{equation}\label{eq:adm}
\Phi(t) := \int_\Omega \nabla z(x,t) \cdot \nabla v(x) \, \d x + \sigma \int_\Omega z(x,t) v(x) \, \d x - \langle f(t), v\rangle_V \leq 0
\end{equation}
for every $v \in V$ satisfying $v \geq 0$ a.e.~in $\Omega$ and for a.e.~$t \in (0,T)$. On the other hand, $\Phi$ is continuous on $[0,T]$ since $z \in C([0,T];V)$ (see Lemma \ref{lem:phi}) and $f \in W^{1,2}(0,T;V^*) \subset C([0,T];V^*)$. Hence \eqref{eq:adm} also holds true for all $t \in [0,T]$; in particular, $\Phi(0) \leq 0$, i.e., $-\Delta z_0 + \sigma z_0 - f(\cdot,0) \leq 0$ in $V^*$.
\end{enumerate}
\end{rem}

We first show existence of strong solutions to \eqref{eq:vi} (or equivalently \eqref{eq}) along with \eqref{bcic}.

\subsection{Time-discretization}\label{ssec:dis}

Let $m \in \N$ and set $\tau=T/m$ and $t_k=k \tau$ for $k=0,1,\dots,m$.
Moreover, we set
\begin{alignat}{4}
f_k &:= \frac{1}{\tau} \int^{t_k}_{t_{k-1}} f(\cdot,s) \,\d s \; \in L^2(\Omega) \quad \mbox{ for } \ k=1,2,\dots,m,\label{eq:defoffk}\\
f_0 &:= f(\cdot,0) \; \in V^*.\label{eq:defoff0}
\end{alignat}
Then the assumption (iii) of Theorem \ref{thm:main} yields
\begin{equation}
\Delta z_0 - \sigma z_0 + f_0 \geq 0 \quad \mbox{ in } \ V^*. \label{ass:ini}
\end{equation}

Let us consider the following discretized problem for \eqref{eq}: For each $ k=1,\dots,m $, find a (strong) solution $z_k \in X \cap V$ of
\begin{equation}\label{eq:dis}
\left\lbrace
\begin{alignedat}{4}
&\disp \dind \left( \frac{z_k-z_{k-1}}{\tau} \right) -\Delta z_k + \sigma z_k \ni f_k & \quad &\mbox{ in } \ \Omega,\\
&z_k =0 &&\mbox{ on } \ \ddom,\\
&\partial_\nu z_k = 0 &&\mbox{ on } \ \ndom,
\end{alignedat}
\right.
\end{equation}
where the inclusion is equivalent to
$$
\disp \partial I_{(-\infty, 0]} \left( z_k-z_{k-1} \right) -\Delta z_k + \sigma z_k \ni f_k \quad \mbox{ in } \ \Omega.
$$
To this end, given ${\psi} \in X \cap V$ and $g \in L^2(\Omega)$, we shall find a function $z \in X \cap V$ solving
\begin{equation}\label{eq:static}
\left\lbrace
\begin{alignedat}{4}
&\dind \left( z-\psi \right) -\Delta z + {\sigma} z \ni g & \quad &\mbox{ in } \ \Omega,\\
&z =0 &&\mbox{ on } \ \ddom,\\
&\partial_\nu z = 0 &&\mbox{ on } \ \ndom.
\end{alignedat}
\right.
\end{equation}

\begin{defi}\label{def:static}
A measurable function $z = z(x)$ is called a \emph{strong solution} to \eqref{eq:static}, if the following (i)--(iii) are all satisfied:
\begin{enumerate}
\item $z \in X \cap V$;
\item $z \leq {\psi} \,$ a.e.~in $\Omega$;
\item there exists $\xi \in L^2(\Omega)$ such that $\xi \in \dind \left( z-{\psi} \right)$ and $\xi = \Delta z - {\sigma} z + g $ a.e.~in $\Omega$.
\end{enumerate}
In particular, $ z \in X \cap V $ is a strong solution to \eqref{eq:static} if and only if it holds that
\begin{equation}\label{eq2:static}
z \leq \psi, \quad - \Delta z + {\sigma} z \leq g, \quad (-\Delta z + {\sigma} z - g) (z - \psi) = 0 \quad \mbox{ a.e.~in } \ \Omega.
\end{equation}
\end{defi}

Thanks to~\cite[Theorems 2.1 and 2.2]{AK}, we have the following proposition concerned with existence and the Lewy-Stampacchia estimate for strong solutions to \eqref{eq:static}:

\begin{prop}\label{pro:static}
For any ${\psi} \in X \cap V$ and $g \in L^2(\Omega)$, there exists a unique strong solution $z = z(x)$ to \eqref{eq:static} such that the so-called Lewy-Stampacchia's inequality holds true\/{\rm :}
\begin{equation}\label{eq:ls}
g \wedge \left( -\Delta {\psi} + {\sigma} {\psi} \right) \leq -\Delta z + {\sigma} z \leq g \quad \mbox{ a.e.~in } \ \Omega.
\end{equation}

Furthermore, let ${\psi}^j \in X \cap V$ and $g^j \in L^2(\Omega)$ for $j=1,2$ satisfy
\[
{\psi}^1 \leq {\psi}^2 \quad \mbox{ and } \quad g^1 \leq g^2 \quad \mbox{ a.e.~in } \ \Omega.
\]
Then the strong solutions $z^j$ to \eqref{eq:static} with ${\psi} = {\psi}^j$ and $g=g^j$ for $j=1,2$ also satisfy
\begin{equation}\label{eq:comp}
z^1 \leq z^2 \quad \mbox{ a.e.~in } \ \Omega.
\end{equation}
\end{prop}

\begin{proof}
Define a functional $J : V \to \R$ by
\[
J(z) = \frac{1}{2} \int_\Omega \left| \nabla z \right|^2 \d x + \frac{{\sigma}}{2} \int_\Omega \left| z \right|^2 \d x - \int_\Omega gz \,\d x \quad \mbox{ for } \ z \in V,
\]
and set
\[
K := \left\{ v \in V \colon v \leq {\psi} \quad \mbox{a.e.~in } \ \Omega \right\}.
\]
Thanks to~\cite[Theorems 2.1 and 2.2]{AK}, there exists a unique minimizer $z$ of $J$ over $K$ such that $z \in X \cap K$; moreover, the Lewy-Stampacchia inequality \eqref{eq:ls} and the comparison principle (see \eqref{eq:comp}) hold true.

We claim that $z \in X \cap K$ minimizes $J$ over $K$ if and only if $ z $ is a solution to \eqref{eq:static}. Indeed, if $z \in X \cap K$ minimizes $J$ over $K$, then 
\begin{equation*}
J(z) \leq J(u) \quad \mbox{ for any } \ u \in K.
\end{equation*}
Since $K$ is convex, it follows that
\begin{align*}
J(z) \leq J \left( (1-\theta) z + \theta v \right) \quad \mbox{ for } \ v \in K \ \mbox{ and } \ \theta \in (0,1),
\end{align*}
which along with the arbitrariness of $\theta \in (0,1)$ implies
\begin{align}
0 &\leq -\int_\Omega \left| \nabla z \right|^2 \d x + \int_\Omega \nabla z \cdot \nabla v \,\d x - {\sigma} \int_\Omega \left| z \right|^2 \d x + {\sigma} \int_\Omega zv \,\d x - \int_\Omega g \left( v-z \right) \d x\notag\\
&= \int_\Omega \left( -\Delta z + {\sigma} z - g \right) \left( v-z \right) \d x \quad \mbox{ for } \ v \in K. \label{eq:star}
\end{align}
Moreover, let $w \in V$ be such that $w \leq 0$ a.e.~in $\Omega$. Substituting $v=z+w \in K$ in \eqref{eq:star}, we obtain
\begin{equation*}
0 \leq \int_\Omega \left( -\Delta z + {\sigma} z - g \right) w \, \d x \quad \mbox{ for } \ w \in V \quad \mbox{ satisfying } \quad w \leq 0 \quad \mbox{ a.e.~in } \ \Omega,
\end{equation*}
which implies that $-\Delta z + {\sigma} z - g \leq 0$ a.e.~in $\Omega$. Furthermore, it follows from \eqref{eq:star} with $v = 2z - \psi \in K$ and that with $v = \psi \in K$ that
\[
\int_\Omega \left( -\Delta z + {\sigma} z - g \right) \left( z-{\psi} \right) \d x = 0.
\]
Since $(-\Delta z + {\sigma} z - g) (z-{\psi})$ is non-negative a.e.~in $\Omega$, we deduce that
\begin{equation}\label{eq:var3}
(-\Delta z + {\sigma} z - g) (z-{\psi}) = 0 \quad \mbox{ a.e.~in } \ \Omega.
\end{equation}
Let $\xi = \Delta z - {\sigma} z + g \in L^2(\Omega)$. Then $\xi \geq 0$ a.e.~in $\Omega$, and by \eqref{eq:var3}, for a.e.~$x \in \Omega$, $\xi(x)=0$ if $z(x)<\psi(x)$. Hence we have $\xi \in \dind (z-\psi)$. Therefore we conclude that $z$ is a solution to \eqref{eq:static}.

We next prove the inverse. Let $z \in X \cap V$ be a strong solution to \eqref{eq:static}. Then we immediately have $z \in K$. Moreover, by virtue of \eqref{eq2:static}, we find that
\begin{align*}
J(w) - J(z) &\geq \int_\Omega \left( -\Delta z + {\sigma} z -g \right) \left( w-z \right) \d x\\
&= \int_\Omega \left( -\Delta z + {\sigma} z -g \right) \left( w-\psi \right) \d x - \int_\Omega \left( -\Delta z + {\sigma} z - g \right) \left( z-\psi \right) \d x\\
&= \int_\Omega \left( -\Delta z + {\sigma} z -g \right) \left( w-\psi \right) \d x \geq 0 \quad \mbox{ for } \ w \in K.
\end{align*}
Therefore $z \in X \cap K$ minimizes $J$ over $K$.

Combining all these facts, one can assure that there exists a unique strong solution to \eqref{eq:static} satisfying \eqref{eq:ls} and \eqref{eq:comp}.
\end{proof}

\subsection{A priori estimates}\label{ssec:bdd}

Owing to Proposition \ref{pro:static}, for each $k=1,\dots,m$, the discretized equation \eqref{eq:dis} admits a unique strong solution $z_k \in X \cap V$ satisfying \eqref{eq:ls} with $\psi = z_{k-1}$ and $g = f_k$. We can then define the \emph{piecewise linear interpolants} $\ztau$ and $\ftau$ of $\{z_k\}_{k=0,\ldots,m}$ and $\{f_k\}_{k=0,\ldots,m}$ by
\begin{alignat}{4}
\ztau(x,t) &:= z_{k-1}(x) + \frac{t-t_{k-1}}{\tau} \left( z_k(x)-z_{k-1}(x) \right) & \quad &\mbox{ for } \ x \in \Omega, \quad t_{k-1} < t \leq t_k,\label{def:ztau}\\
\ztau(x,0) &:= z_0(x) &&\mbox{ for } \ x \in \Omega,\notag\\
\ftau(x,t) &:= f_{k-1}(x) + \frac{t-t_{k-1}}{\tau} \left( f_k(x)-f_{k-1}(x) \right) &&\mbox{ for } \ x \in \Omega, \quad t_{k-1} < t \leq t_k,\label{def:ftau}\\
\ftau(\cdot, 0) &:= f_0 \; \in V^*,\notag
\end{alignat}
respectively, for $k=1,\dots,m$, and define the \emph{piecewise constant interpolants} $\ztaubar$ and $\ftaubar$ of $\{z_k\}_{k=1,\ldots,m}$ and $\{f_k\}_{k=1,\ldots,m}$ by
\begin{alignat}{4}
\ztaubar(x,t) &:= z_k(x) & \quad &\mbox{ for } \ x \in \Omega, \quad t_{k-1} < t \leq t_k,\label{def:ztaubar}\\
\ftaubar(x,t) &:= f_k(x) &&\mbox{ for } \ x \in \Omega, \quad t_{k-1} < t \leq t_k,\label{def:ftaubar}
\end{alignat}
respectively, for $k=1,\dots,m$.
Then it follows that
\begin{alignat*}{4}
\ztau &\in W^{1,2} (0,T;X \cap V), & \ \quad \ &\ztaubar &&\in L^\infty (0,T;X \cap V),\\
\ftau &\in W^{1,2} (0,T;V^*), &&\ftaubar &&\in L^\infty (0,T;L^2(\Omega))
\end{alignat*}
for each $ \tau $. Moreover,
\eqref{eq:dis} is rephrased as
\begin{equation*}
\left\lbrace
\begin{alignedat}{4}
&\dind \left( \partial_t \ztau \right) - \Delta \ztaubar + \sigma \ztaubar \ni \ftaubar & \quad &\mbox{ in } \ \Omega \times (0,T),\\
&\ztau = \ztaubar = 0 &&\mbox{ on } \ \ddom \times (0,T),\\
&\partial_\nu \ztau = \partial_\nu \ztaubar = 0 & \quad &\mbox{ on } \ \ndom \times (0,T).
\end{alignedat}
\right.
\end{equation*}

We shall establish a priori estimates for $(\ztau)$ and $(\ztaubar)$.

\begin{lem}\label{lem:ztau}
There exists a constant $C>0$ independent of $\tau$ and $T$ such that
\begin{equation}\label{eq:ztau}
\left\| \partial_t \ztau \right\|_{L^2(0,T;V)} \leq C \left\| \partial_t f \right\|_{L^2(0,T;V^*)} \quad \mbox{ for } \ \tau \in (0,1).
\end{equation}
Moreover, $(\ztau)$ is also bounded in $L^\infty(0,T;V)$ for $\tau\in(0,1)$.
\end{lem}

\begin{proof}
For each $k=1, \dots, m$, set $ \eta_k := \Delta z_k - \sigma z_k + f_k \in L^2(\Omega) $. Then we see that
\begin{equation}\label{eq:etak}
\eta_k \in \dind \left( \frac{z_k-z_{k-1}}{\tau} \right) \quad \mbox{ a.e.~in } \ \Omega
\end{equation}
for $ k=1,\dots,m $.
Moreover, set $\eta_0 := \Delta z_0 -\sigma z_0 + f_0 \geq 0$ (see \eqref{ass:ini}). By subtraction, we see that
\begin{equation}\label{eq:subtr}
\frac{\eta_k - \eta_{k-1}}{\tau} - \Delta \left( \frac{z_k - z_{k-1}}{\tau} \right) + \sigma \, \frac{z_k - z_{k-1}}{\tau} = \frac{f_k - f_{k-1}}{\tau}
\end{equation}
for each $k = 1,\ldots,m$. Testing \eqref{eq:subtr} by $\left( z_k-z_{k-1} \right) / \tau$ $\in V$, we have
\begin{align}
\mel{\int_\Omega \left(\frac{\eta_k-\eta_{k-1}}{\tau}\right) \left(\frac{z_k-z_{k-1}}{\tau} \right) \d x + \left\| \nabla \left( \frac{z_k-z_{k-1}}{\tau} \right) \right\|^2_{L^2(\Omega)} + \sigma \left\| \frac{z_k-z_{k-1}}{\tau} \right\|^2_{L^2(\Omega)}} \notag\\
&= \int_\Omega \left(\frac{f_k - f_{k-1}}{\tau}\right)\left(\frac{z_k-z_{k-1}}{\tau}\right) \d x \label{eq:diff}
\end{align}
for $k=1,\dots,m$.
By virtue of \eqref{eq:etak} and $\eta_0 \geq 0$ a.e.~in $ \Omega $, we can deduce that
\[
\int_\Omega \left(\frac{\eta_k-\eta_{k-1}}{\tau}\right) \left(\frac{z_k-z_{k-1}}{\tau}\right) \d x = -\int_\Omega \left(\frac{\eta_{k-1}}{\tau}\right) \left(\frac{z_k-z_{k-1}}{\tau}\right) \d x \geq 0
\]
for $k = 1,\ldots,m$.
Indeed, we observe that $ \eta_k (z_k - z_{k-1}) = 0 $, $ \eta_{k-1} \geq 0 $ and $ z_k \leq z_{k-1} $ a.e.~in $ \Omega $.
Therefore (with the aid of Poincar\'{e}'s inequality when $\sigma=0$) we derive from \eqref{eq:diff} that
\begin{equation}\label{eq:zkdiff}
\left\| \frac{z_k-z_{k-1}}{\tau} \right\|^2_V \leq C \left\| \frac{f_k-f_{k-1}}{\tau} \right\|^2_{V^*} \quad \mbox{ for } \ k=1,\dots,m,
\end{equation}
where $C$ is a positive constant independent of $k$, $\tau$ and $T$. Moreover, one can observe that
\begin{equation}\label{est:inter}
\tau \sum^m_{k=1} \left\| \frac{f_k-f_{k-1}}{\tau} \right\|^2_{V^*} \leq 4 \int^T_0 \left\| \partial_t f(s) \right\|^2_{V^*} \d s.
\end{equation}
For the convenience of the reader, we shall give a proof for \eqref{est:inter} in Proposition \ref{pro:inter} of Appendix \ref{sec:intpl}. Noting that $\partial_t \ztau(t) = (z_k - z_{k-1}) / \tau$ for $t \in (t_{k-1}, t_k)$, we deduce from \eqref{eq:zkdiff} and \eqref{est:inter} that
\[
\left\| \partial_t \ztau \right\|^2_{L^2(0,T;V)} = \tau \sum^m_{k=1} \left\| \frac{z_k-z_{k-1}}{\tau} \right\|^2_{V} \leq C \tau \sum^m_{k=1} \left\| \frac{f_k - f_{k-1}}{\tau} \right\|^2_{V^*} \leq C \left\| \partial_t f \right\|^2_{L^2(0,T;V^*)}.
\]
Thus \eqref{eq:ztau} follows. Furthermore, the boundedness of $(\partial_t \ztau)$ in $L^2(0,T;V)$ along with $\ztau(0) = z_0$ implies that $(\ztau)$ is bounded in $L^\infty(0,T;V)$; indeed, it follows that
\[
\left\| \ztau(t) \right\|_V \leq \int^t_0 \left\| \partial_t \ztau(s) \right\|_V \d s + \left\| z_0 \right\|_V \leq T^{1/2} \left\| \partial_t \ztau \right\|_{L^2(0,T;V)} + \left\| z_0 \right\|_V
\]
for $t \in [0,T]$ and $\tau \in (0,1)$.
This completes the proof.
\end{proof}

We further prove that

\begin{lem}\label{lem:est2}
There exists a constant $C > 0$ independent of $\tau \in (0,1)$ such that
\begin{equation}\label{eq:ztaubar}
\left\| \ztaubar \right\|_{L^\infty(0,T; V)} + \left\| \ztaubar \right\|_{L^2(0,T;X)} + \|z_\tau\|_{L^2(0,T;X)} \leq C.
\end{equation}
In particular, it holds that
\begin{equation}\label{eq:pointwise}
\left\| -\Delta \ztaubar(t) + \sigma \ztaubar(t) \right\|^2_{L^2(\Omega)} \leq 2 \left( \| \hat{f} \wedge (-\Delta z_0 + \sigma z_0) \|^2_{L^2(\Omega)} + \| \ftaubar(t) \|^2_{L^2(\Omega)} \right)
\end{equation}
for a.e.~$ t \in (0,T) $ and all $ \tau \in (0,1) $.
\end{lem}

\begin{proof}
First we shall show the boundedness of $(\ztaubar)$ in $L^\infty(0,T;V)$. To this end, we observe that
\begin{align*}
\left\| \ztau(t) - \ztaubar(t) \right\|_V &= \left\| z_{k-1} + \frac{t-t_{k-1}}{\tau} \left( z_k - z_{k-1} \right) - z_k \right\|_V\\
&= \left| t-t_k \right| \left\| \frac{z_k-z_{k-1}}{\tau} \right\|_V\\
&\hspace{-2.1truemm}\stackrel{\text{\eqref{eq:zkdiff}}}{\leq} C \tau \left\| \frac{f_k-f_{k-1}}{\tau} \right\|_{V^*} \stackrel{\text{\eqref{est:inter}}}{\leq} C \tau^{1/2} \left\| \partial_t f \right\|_{L^2(0,T;V^*)}
\end{align*}
for each $t \in (t_{k-1}, t_k]$ and $k = 1,\ldots,m$.
Hence we infer that
\begin{equation}\label{eq:intpldif}
\left\| \ztau - \ztaubar \right\|_{L^\infty(0,T;V)} \to 0 \quad \mbox{ as } \ \tau \to 0_+. 
\end{equation}
In particular, since $(\ztau)$ is bounded in $L^\infty(0,T;V)$ by virtue of Lemma \ref{lem:ztau}, so is $(\ztaubar)$.

We shall next estimate $\Delta \ztaubar$ by making use of the Lewy-Stampacchia inequality,
\begin{equation}\label{eq:lszk}
f_k \wedge \left( -\Delta z_{k-1} + \sigma z_{k-1} \right) \leq -\Delta z_k + \sigma z_k \leq f_k \quad \mbox{ a.e.~in } \ \Omega,
\end{equation}
which is derived from \eqref{eq:ls} with $\psi = z_{k-1}$ and $g = f_k$ for each $k=1,\dots,m$. Hence using the assumption (ii) of Theorem \ref{thm:main}, we deduce from \eqref{eq:lszk} that 
\begin{align*}
f_k \geq -\Delta z_k + \sigma z_k &\geq f_k \wedge \left( -\Delta z_{k-1} + \sigma z_{k-1} \right)\\
&\geq f_k \wedge \big\{ f_{k-1} \wedge \left( -\Delta z_{k-2} + \sigma z_{k-2} \right) \big\}\\
&\geq f_k \wedge f_{k-1} \wedge \dots \wedge f_1 \wedge \left( -\Delta z_0 + \sigma z_0 \right)\\
&\geq \hat{f} \wedge \left( -\Delta z_0 + \sigma z_0 \right) \notag,
\end{align*}
which yields
\[
\left| -\Delta z_k + \sigma z_k \right| \leq | \hat{f} \wedge \left( -\Delta z_0 + \sigma z_0 \right)| + |f_k|.
\]
Thus we obtain
\begin{align*}
\left\| -\Delta \ztaubar(t) + \sigma \ztaubar(t) \right\|^2_{L^2(\Omega)} &\leq 2 \left( \| \hat{f} \wedge \left( -\Delta z_0 + \sigma z_0 \right) \|^2_{L^2(\Omega)} + \| \ftaubar(t) \|^2_{L^2(\Omega)} \right)
\end{align*}
for a.e.~$ t \in (0,T) $ and all $ \tau>0 $. Hence \eqref{eq:pointwise} is proved.
Moreover, it also follows that
\begin{align*}
\mel{\left\| \Delta \ztaubar \right\|^2_{L^2(0,T;L^2(\Omega))}}\\
&\leq C \left( T \| \hat{f} \|^2_{L^2(\Omega)} + T \left\| -\Delta z_0 + \sigma z_0 \right\|^2_{L^2(\Omega)} + \sigma^2 \left\| \ztaubar \right\|^2_{L^2(0,T;L^2(\Omega))} + \| \ftaubar \|^2_{L^2(0,T;L^2(\Omega))} \right),
\end{align*}
where $ C $ is a positive constant independent of $ \tau $ and $ T $.
Noting that
\begin{align*}
\sigma^2 \left\| \ztaubar \right\|^2_{L^2(0,T;L^2(\Omega))} + \| \ftaubar \|^2_{L^2(0,T;L^2(\Omega))} \leq T \sigma^2 \left\| \ztaubar \right\|_{L^\infty(0,T;L^2(\Omega))}^2 + \| \ftaubar \|^2_{L^2(0,T;L^2(\Omega))} \leq C
\end{align*}
(see also Proposition \ref{pro:f} in Appendix \ref{sec:intpl} for the boundedness of $ (\ftaubar) $),
we obtain
\[
\left\| \Delta \ztaubar \right\|^2_{L^2(0,T;L^2(\Omega))} \leq C.
\]
Furthermore, we find by convexity that
\[
\left\| \Delta z_\tau(t) \right\|_{L^2(\Omega)}^2 \leq \frac{t_k-t}{\tau} \left\| \Delta z_{k-1} \right\|_{L^2(\Omega)}^2 + \frac{t-t_{k-1}}{\tau} \left\| \Delta z_k \right\|_{L^2(\Omega)}^2 \quad \mbox{ for a.e.~} t \in (t_{k-1},t_k),
\]
which yields
\[
\int^{t_k}_{t_{k-1}} \|\Delta z_\tau(t)\|_{L^2(\Omega)}^2 \, \d t \leq \frac{\tau}2 \|\Delta z_{k-1}\|_{L^2(\Omega)}^2 + \frac{\tau}2 \|\Delta z_k\|_{L^2(\Omega)}^2.
\]
It follows that
\begin{align*}
\int^T_0 \|\Delta z_\tau(t)\|_{L^2(\Omega)}^2 \, \d t &= \sum_{k=1}^m \int^{t_k}_{t_{k-1}} \|\Delta z_\tau(t)\|_{L^2(\Omega)}^2 \, \d t\\
&\leq \frac{\tau}2 \sum_{k=1}^m \|\Delta z_{k-1}\|_{L^2(\Omega)}^2 + \frac{\tau}2 \sum_{k=1}^m\|\Delta z_k\|_{L^2(\Omega)}^2\\
&\leq \frac{\tau}2 \|\Delta z_0\|_{L^2(\Omega)}^2 + \int^T_0 \|\Delta \ztaubar(t)\|_{L^2(\Omega)}^2 \, \d t \leq C.
\end{align*}
This completes the proof.
\end{proof}

Combining Lemmas \ref{lem:ztau} and \ref{lem:est2}, we can derive

\begin{lem}\label{lem:strong}
The sequence of the piecewise linear interpolants $(\ztau)$ is relatively compact in $C([0,T];V)$.
\end{lem}

\begin{proof}
We first prove the equi-continuity of $(\ztau)$ in $C([0,T];V)$. Using Lemma \ref{lem:ztau}, one observes that
\begin{align*}
\left\| \ztau(t) - \ztau(s) \right\|_V &\leq \int^t_s \left\| \partial_t \ztau(r) \right\|_V \d r\\
&\leq \left| t-s \right|^{1/2} \left\| \partial_t \ztau \right\|_{L^2(0,T;V)} \leq C \left| t-s \right|^{1/2}
\end{align*}
for any $t ,s \in [0,T]$. Hence since $(\ztau)$ is bounded in $L^2(0,T;X)$ (see \eqref{eq:ztaubar} in Lemma \ref{lem:est2}) and $X$ is compactly embedded in $V$, using Theorem 3 of~\cite{Simon}, one can conclude that $(\ztau)$ is relatively compact in $C([0,T];V)$. This completes the proof.
\end{proof}

\subsection{Identification of weak limits}\label{ssec:id}

Thanks to Lemmas \ref{lem:ztau}, \ref{lem:est2} and \ref{lem:strong}, there exist a subsequence $(\tau')$ of $(\tau)$ and $z \in W^{1,2}(0,T;V) \cap L^2 (0,T;X)$ ($\subset C([0,T];V)$) such that
\begin{alignat*}{4}
z_{\tau'} &\to z & \quad &\mbox{ weakly in } \ W^{1,2}(0,T;V) \cap L^2(0,T;X),\\
& &&\mbox{ strongly in } \ C([0,T];V),\\
\overline{z}_{\tau'} &\to z &&\mbox{ weakly in } \ L^2 (0,T;X),\\
& &&\mbox{ strongly in } \ L^\infty(0,T;V)
\end{alignat*}
as $\tau' \to 0_+$. Here we used the fact from \eqref{eq:intpldif} that $\ztau-\ztaubar \to 0$ strongly in $L^\infty (0,T;V)$ as $\tau \to 0_+$. We are now in a position to prove the existence part of Theorem \ref{thm:main}.

\begin{proof}[Proof of Theorem {\rm \ref{thm:main}} {\rm (}The existence part\/{\rm )}]
We shall verify that $z$ satisfies \eqref{eq:A} as well as (i)--(iv) of Definition \ref{D:sol}.	We have already proved that $z \in W^{1,2}(0,T;V) \cap L^2(0,T;X)$, which in particular implies (i) and \eqref{eq:A}. By virtue of Lemma \ref{lem:strong}, we have
\begin{align*}
\left\| z(0)-z_0 \right\|_V &= \left\| z(0)-z_{\tau'}(0) \right\|_V \leq \sup_{t \in [0,T]} \left\| z(t) - z_{\tau'}(t) \right\|_V \to 0 \quad \mbox{ as } \ \tau' \to 0_+,
\end{align*}
which assures the initial condition (iv).

It still remains to check the conditions (ii) and (iii). Since $z_k$ is a strong solution to \eqref{eq:dis} for each $ k $, we deduce from \eqref{eq2:static} that
\begin{alignat}{4}
-\Delta \ztaubar + \sigma \ztaubar - \ftaubar \leq 0, \quad \partial_t \ztau &\leq 0 & \quad &\mbox{ a.e.~in } \ \Omega \times (0,T), \notag \\
\left( -\Delta \ztaubar + \sigma \ztaubar - \ftaubar \right) \partial_t \ztau &= 0 &&\mbox{ a.e.~in } \ \Omega \times (0,T). \label{eq:eqtau}
\end{alignat}
Note that
\begin{equation}\label{ftau-conv}
\ftaubar \to f \quad \mbox{ strongly in } \ L^2 (0,T;L^2(\Omega))
\end{equation}
as $\tau \to 0_+$ (see Proposition \ref{pro:f} in Appendix \ref{sec:intpl}). From the (weak) convergences obtained so far, it follows immediately that
\begin{equation}\label{eq:ineq}
-\Delta z + \sigma z - f \leq 0, \quad \partial_t z \leq 0 \quad \mbox{ a.e.~in } \ \Omega \times (0,T),
\end{equation}
which implies (ii). To prove (iii), we claim that
\[
0 \leq \int^T_0 \int_\Omega \left( \Delta z -\sigma z + f \right) \partial_t z \, \d x \, \d t.
\]
Indeed, recalling \eqref{def:ztau} and \eqref{def:ztaubar}, we observe that
\begin{align*}
\int^T_0 \int_\Omega \left( \Delta \overline{z}_{\tau} \right) \partial_t z_{\tau} \, \d x \,\d t 
&= \tau \sum^m_{k=1} \int_\Omega \Delta z_k \left( \frac{z_k - z_{k-1}}{\tau} \right) \d x\\
&\leq \sum^m_{k=1} \int_\Omega \left( -\frac{1}{2} \left| \nabla z_k \right|^2 + \frac{1}{2} \left| \nabla z_{k-1} \right|^2 \right) \d x\\
&= -\phi \left( z_{\tau}(T) \right) + \phi \left( z_0 \right),
\end{align*}
where $\phi : L^2(\Omega) \to (-\infty, +\infty]$ is the functional defined by \eqref{eq:phi}. Therefore, by the continuity of $\phi$ in $V$ and Lemma \ref{lem:phi}, since $z_{\tau'}(T) \to z(T)$ strongly in $V$, it follows that
\begin{align*}
\limsup_{\tau' \to 0_+} \int^T_0 \int_\Omega \left( \Delta \overline{z}_{\tau'} \right) \partial_t z_{\tau'} \, \d x \,\d t 
&\leq \limsup_{\tau' \to 0_+} \left( -\phi \left( z_{\tau'}(T) \right) + \phi (z_0) \right)\\
&= -\phi (z(T)) + \phi(z_0) = -\int^T_0 \frac{\d}{\d t} \, \phi(z(t)) \,\d t\\
&\hspace{-1.5truemm}\stackrel{\text{\eqref{eq:chain}}}{=} \int^T_0 \int_\Omega \left(\Delta z\right) \partial_t z \, \d x \, \d t.
\end{align*}
Moreover, we find from \eqref{ftau-conv} that
\[
\lim_{\tau' \to 0_+} \int^T_0 \int_\Omega \overline{f}_{\tau'} \partial_t z_{\tau'} \,\d x \,\d t = \int^T_0 \int_\Omega f \partial_t z \, \d x \, \d t.
\]
Furthermore, we see that
\begin{align*}
\sigma \int^T_0 \int_\Omega \overline{z}_{\tau'} \partial_t z_{\tau'} \, \d x \, \d t &= \sigma \sum^m_{k=1} \int_\Omega z_k \left( z_k - z_{k-1} \right) \d x \\
&\geq \frac{\sigma}{2} \sum^m_{k=1} \int_\Omega \left( \left| z_k \right|^2 -  \left| z_{k-1} \right|^2 \right) \d x\\
&= \frac{\sigma}{2} \int_\Omega \left| z_{\tau'}(T) \right|^2 \d x - \frac{\sigma}{2} \int_\Omega \left| z_0 \right|^2 \d x,
\end{align*}
which implies that
\begin{align*}
\mel{
\liminf_{\tau' \to 0_+} \left( \sigma \int^T_0 \int_\Omega \overline{z}_{\tau'}  \partial_t z_{\tau'} \, \d x \, \d t \right)
}\\
&\geq \frac{\sigma}{2} \int_\Omega \left| z(T) \right|^2 \d x - \frac{\sigma}{2} \int_\Omega \left| z_0 \right|^2 \d x\\
&= \frac{\sigma}{2} \int^T_0 \frac{\d}{\d t} \left( \int_\Omega \left| z(t) \right|^2 \d x \right) \d t = \sigma \int^T_0 \int_\Omega z \partial_t z \, \d x \, \d t.
\end{align*}
Combining all these facts, we obtain
\begin{align*}
0 &\stackrel{\eqref{eq:eqtau}}{=}\limsup_{\tau' \to 0_+} \int^T_0 \int_\Omega \left( \Delta \overline{z}_{\tau'} - \sigma \overline{z}_{\tau'} + \overline{f}_{\tau'} \right) \partial_t z_{\tau'} \, \d x \, \d t\\
&\hspace{2.1truemm}\leq \int^T_0 \int_\Omega \left( \Delta z -\sigma z + f \right) \partial_t z \, \d x \, \d t \stackrel{\eqref{eq:ineq}}\leq 0.
\end{align*}
Hence (iii) follows, for the integrand is non-positive a.e.~in $\Omega \times (0,T)$ (see \eqref{eq:ineq}). Thus the existence part of Theorem \ref{thm:main} has been proved.
\end{proof}

\subsection{Continuous dependence and uniqueness}\label{ssec:uniq}

In this subsection, in order to complete the proof of Theorem \ref{thm:main}, we shall prove uniqueness and continuous dependence on prescribed data $f$ and $z_0$ of strong solutions to the initial-boundary value problem.

\begin{proof}[Proof of Theorem {\rm \ref{thm:main}} {\rm (}The uniqueness part\/{\rm )}]
Let $z^j$ be a strong solution to \eqref{eq} (or \eqref{eq:vi}) and \eqref{bcic} with $z_0=z_0^j$ and $f=f^j$ for $j=1,2$.
Set $\eta^j = \Delta z^j - \sigma z^j + f^j \in \dind (\partial_t z^j)$ for $j=1,2$. By subtraction, we see that
\begin{alignat*}{4}
\eta^1 - \eta^2 - \Delta \left( z^1 - z^2 \right) + \sigma \left( z^1 - z^2 \right) &= f^1 - f^2 & \quad &\mbox{ a.e.~in } \ \Omega \times (0,T),\\
z^1 - z^2 &= 0 &&\mbox{ a.e.~on } \ \ddom \times (0,T),\\
\partial_\nu \left( z^1 - z^2 \right) &= 0 &&\mbox{ a.e.~on } \ \ndom \times (0,T),\\
z^1(0) - z^2(0) &= z^1_0 - z^2_0 &&\mbox{ a.e.~in } \ \Omega.
\end{alignat*}
Testing this by $\partial_t z^1(t) - \partial_t z^2(t)$, we have
\begin{align}
\mel{\int_\Omega \left( \eta^1(t)-\eta^2(t) \right) \left( \partial_t z^1(t)-\partial_t z^2(t) \right) \d x + \frac{1}{2} \frac{\d}{\d t} \left\| \nabla z^1(t) - \nabla z^2(t) \right\|^2_{L^2(\Omega)}} \notag \\
&+ \frac{\sigma}{2} \frac{\d}{\d t} \left\| z^1(t) - z^2(t) \right\|^2_{L^2(\Omega)} = \int_\Omega \left( f^1(t) - f^2(t) \right) \left( \partial_t z^1(t) - \partial_t z^2(t) \right) \d x \label{eq:eta1eta2}
\end{align}
for a.e.~$t \in (0,T)$ (see Lemma \ref{lem:phi} again). The first term of the left-hand-side is non-negative from the monotonicity of $\dind$. Hence, \eqref{eq:eta1eta2} yields
\begin{align*}
\mel{\frac{1}{2} \frac{\d}{\d t} \left( \left\| \nabla z^1(t) - \nabla z^2(t) \right\|^2_{L^2(\Omega)} + \sigma \left\| z^1(t)-z^2(t) \right\|^2_{L^2(\Omega)} \right)}\\
&\leq \int_\Omega \left( f^1(t) - f^2(t) \right) \left( \partial_t z^1(t) - \partial_t z^2(t) \right) \d x\\
&= \frac{\d}{\d t} \left\langle f^1(t)-f^2(t) , z^1(t)-z^2(t) \right\rangle_V - \left\langle \partial_t f^1(t) - \partial_t f^2(t), z^1(t)-z^2(t) \right\rangle_V.
\end{align*}
To verify the last equality, see Proposition \ref{pro:leibniz} in Appendix \ref{sec:leibniz}. Moreover, we have
\begin{align*}
\mel{\frac{1}{2} \left\| \nabla z^1(t) - \nabla z^2(t) \right\|^2_{L^2(\Omega)} + \frac{\sigma}{2} \left\| z^1(t)-z^2(t) \right\|^2_{L^2(\Omega)}}\\
&\leq \frac{1}{2} \left\| \nabla z^1_0 - \nabla z^2_0 \right\|^2_{L^2(\Omega)} + \frac{\sigma}{2} \left\| z^1_0-z^2_0 \right\|^2_{L^2(\Omega)}\\
&\quad + \left\langle f^1(t)-f^2(t) , z^1(t)-z^2(t) \right\rangle_V - \left\langle f^1(0)-f^2(0) , z^1_0-z^2_0 \right\rangle_V\\
&\quad - \int^t_0 \left\langle \partial_t f^1(s) - \partial_t f^2(s), z^1(s)-z^2(s) \right\rangle_V \, \d s
\end{align*}
for all $t \in [0,T]$. Then (with the aid of Poincar\'{e}'s inequality when $\sigma=0$) we obtain
\begin{align*}
\left\| z^1(t)-z^2(t) \right\|^2_V &\leq C \Big( \left\| z^1_0 - z^2_0 \right\|^2_V + \left\| f^1(t) - f^2(t) \right\|^2_{V^*} + \left\| f^1(0) - f^2(0) \right\|^2_{V^*}\\
&\quad + \int^t_0 \left\| \partial_t f^1(s) - \partial_t f^2(s) \right\|_{V^*} \left\| z^1(s)-z^2(s) \right\|_V \d s \Big)\\
&\leq C \Big( \left\| z^1_0 - z^2_0 \right\|^2_V + \sup_{t \in [0,T]} \left\| f^1(t)-f^2(t) \right\|^2_{V^*}\\
&\quad + \int^t_0 \left\| \partial_t f^1(s) - \partial_t f^2(s) \right\|_{V^*} \left\| z^1(s)-z^2(s) \right\|_V \d s \Big),
\end{align*}
which along with some variant of Gronwall's inequality (see Lemma A.5 of~\cite{HBF}) implies
\begin{align*}
\MoveEqLeft{
\sup_{t \in [0,T]} \left\| z^1(t)-z^2(t) \right\|_V
}\\
&\leq C \Big( \left\| z^1_0-z^2_0 \right\|_V + \sup_{t \in [0,T]} \left\| f^1(t)-f^2(t) \right\|_{V^*} + \left\| \partial_t f^1 - \partial_t f^2 \right\|_{L^1(0,T;V^*)} \Big).
\end{align*}
Here we note that $C$ may depend on $ \sigma $ and $ \Omega $ (via Poincar\'{e}'s inequality) but it is independent of $T$. Thus one can verify \eqref{conti-dep}, and in particular, uniqueness of strong solution follows. We have proved Theorem \ref{thm:main}.
\end{proof}

\section{Qualitative properties of strong solutions}\label{sec:qua}

In this section, we shall discuss qualitative properties of strong solutions $z=z(x,t)$ for \eqref{eq:vi} (or \eqref{eq}) and \eqref{bcic}. In particular, Theorem \ref{thm:qua} will be proved.

\subsection{Comparison principle}\label{ssec:comp}

We shall prove comparison principle for strong solutions to \eqref{eq:vi} (or \eqref{eq}), \eqref{bcic}.

\begin{prop}[Comparison principle]\label{pro:comp}
Assume that \eqref{eq:ass} is satisfied. Let $\sigma \geq 0$, $f^j \in L^2 (0,T;L^2(\Omega)) \cap W^{1,2} (0,T;V^*)$ and $z^j_0 \in X \cap V$ fulfill all the assumptions {\rm (i)--(iii)} of Theorem {\rm \ref{thm:main}} with $f = f^j$ and $z_0 = z_0^j$ for $j=1,2$, respectively. Let $z^j=z^j(x,t)$ be the strong solutions to \eqref{eq:vi} {\rm (}or equivalently \eqref{eq}{\rm )} and \eqref{bcic} with $f=f^j$ and $z_0=z_0^j$ for $ j=1,2 $. Assume that
\[
f^1 \leq f^2 \quad \mbox{ a.e.~in } \ \Omega \times (0,T) \quad \mbox{ and } \quad z^1_0 \leq z^2_0 \quad \mbox{ a.e.~in } \ \Omega.
\]
Then it holds that
\[
z^1 \leq z^2 \quad \mbox{ a.e.~in } \ \Omega \times (0,T).
\]
\end{prop}

\begin{proof}
From the proof of Theorem \ref{thm:main}, for each $j=1,2$, the unique strong solution $z^j=z^j(x,t)$ is obtained as a limit in $C([0,T];V)$ of the piecewise linear interpolant $z^j_\tau$ of solutions $\{z^j_k\}_{k=0,\ldots,m}$ to the corresponding discretized problems \eqref{eq:dis}. Since one can check easily that $f^1_k \leq f^2_k$ for $k=1,2,\dots,m$, where $\{f^j_k\}_{k=1,\ldots,m}$ is defined by \eqref{eq:defoffk} for $f=f^j$, applying Proposition \ref{pro:static} iteratively, we may deduce that
\[
z^1_k \leq z^2_k \quad \mbox{ a.e.~in } \ \Omega \quad \mbox{ for } \ k=1,\dots,m.
\]
Indeed, for the $k$-th step of the iteration, all the assumptions of Proposition \ref{pro:static} are satisfied with $g=f^j_k \in L^2(\Omega)$ and $\psi=z^j_{k-1} \in X \cap V$ for $j=1,2$, provided that $ z^1_{k-1} \leq z^2_{k-1} $ a.e.~in $ \Omega $ (cf.~$ z^1_0 \leq z^2_0 $ by assumption). Therefore applying the comparison principle stated in Proposition \ref{pro:static}, we infer that
\[
z^1_\tau(t) \leq z^2_\tau(t) \quad \mbox{ a.e.~in } \ \Omega
\]
for all $t \in [0,T]$. Thus we can obtain the assertion by passing to the limit as $\tau \to 0_+$ and by recalling the fact that $z^j_\tau \to z^j$ strongly in $C([0,T];V)$ for $j=1,2$.
\end{proof}

\subsection{Three intrinsic properties}\label{ssec:qua}

We next give a proof of Theorem \ref{thm:qua}. 
\begin{proof}[Proof of Theorem {\rm \ref{thm:qua}}]
The irreversibility (i) follows immediately from \eqref{eq:vi} (or \eqref{eq}). Hence we shall prove (ii) and (iii) below. Let $z = z(x,t)$ be the unique strong solution to \eqref{eq:vi} (or \eqref{eq}), \eqref{bcic}. In particular, (for each representative) there exists a negligible set $N_0 \subset [0,T]$ (i.e., $|N_0| = 0$) such that $z$ satisfies
\[
-\Delta z(t) + \sigma z(t) -f(t) \leq 0 \quad \mbox{ a.e.~in } \ \Omega
\]
for all $t \in [0,T] \setminus N_0$. Let $t \in [0,T] \setminus N_0$ and let $v \in V$ be such that $v \leq z(t)$ a.e.~in $\Omega$. Then we see that
\begin{align*}
0 \leq \int_\Omega \left( -\Delta z(t) + \sigma z(t) - f(t) \right) \left( v-z(t) \right) \d x
\leq \E ( v,t ) - \E ( z(t),t ),
\end{align*}
where $ \E : V \times [0,T] \to \R $ is the energy functional defined by \eqref{eq:e}.
Furthermore, since the function $t \mapsto \E(z(t),t)$ is (absolutely) continuous on $[0,T]$ (see Lemma \ref{lem:phi}) and $ z \in C([0,T];V) $, we can verify that
\[
\E ( z(t),t ) \leq \E ( v,t ) \quad \mbox{ for any } \ t \in [0,T],
\]
provided that $v \in V$ and $v \leq z(t)$ a.e.~in $\Omega$.
Indeed, let $ t_0 \in [0,T] $ be fixed and let $ v \in V $ be such that $ v \leq z(t_0) $ a.e.~in $ \Omega $. One can take a subsequence $ (t_n) $ in $ (0,t_0) $ such that the above relation holds at each $ t=t_n $ and $ t_n \nearrow t_0 $. Then noting by (i) and $ t_n < t_0 $ that $ v \leq z(t_0) \leq z(t_n) $ a.e.~in $ \Omega $, we infer that
\[
\E(z(t_n),t_n) \leq \E(v,t_n) \quad \mbox{ for all } \ n \in \N.
\]
Hence employing the fact that $ z \in C([0,T];V) $ and $ f \in W^{1,2}(0,T;V^*) \subset C([0,T];V^*) $ and passing to the limit as $ n \to \infty $, we obtain the desired relation at $ t=t_0 $.
Thus (ii) follows.

As for (iii), differentiating the function $t \mapsto \E(z(t),t)$ and using Lemma \ref{lem:phi}, we have
\begin{align*}
\frac{\d}{\d t} \, \E ( z(t),t ) &= \int_\Omega \left( -\Delta z(t) + \sigma z(t) -f(t) \right) \partial_t z(t) \,\d x - \left\langle \partial_t f(t), z(t) \right\rangle_V \\
&= -\left\langle \partial_t f(t), z(t) \right\rangle_V
\end{align*}
for a.e.~$t \in (0,T)$. Here we also used (iii) of Definition \ref{D:sol}. Hence it follows that
\[
\E ( z(t),t ) - \E ( z(s),s ) = -\int^t_s \left\langle \partial_t f(r), z(r) \right\rangle_V \d r
\]
for all $t,s \in [0,T]$. The proof is completed.
\end{proof}

\section{Global well-posedness and long-time dynamics}\label{sec:global}

In this section, we shall first prove existence of global-in-time solutions for the initial-boundary value problem \eqref{eq:vi} (or \eqref{eq}), \eqref{bcic} (see Theorem \ref{thm:global}), and then, long-time behavior of global-in-time solutions will be discussed. More precisely, we shall verify that each global-in-time solution converges to a stationary limit as $t \to \infty$ under certain assumptions for $f$ (see Theorem \ref{thm:asympt}).

\begin{proof}[Proof of Theorem {\rm \ref{thm:global}}]
For each $T>0$, set $f_T := \left. f \right|_{(0,T)} \in L^2(0,T;L^2(\Omega)) \cap W^{1,2}(0,T;V^*)$. Then thanks to Theorem \ref{thm:main} there exists a strong solution $z_T \in W^{1,2}(0,T;V) \cap L^2(0,T;X)$ to \eqref{eq:vi}, \eqref{bcic} on $[0,T]$. From the uniqueness of strong solution, for any $T' > T$, $z_T$ coincides with $z_{T'}$ in $\Omega \times (0,T)$. Therefore one can define $z \in W^{1,2}_{\rm loc}(0,\infty;V) \cap L^2_\mathrm{loc}([0,\infty);X)$ as an extension of $(z_T)$ onto $[0,\infty)$.
Then, from the arbitrariness of $T>0$, the function $z = z(x,t)$ turns out to be a global-in-time strong solution to \eqref{eq:vi} (or \eqref{eq}) and \eqref{bcic}.
Moreover, estimate \eqref{eq:ztau} implies that
\[
\left\| \partial_t z_T \right\|_{L^2(0,T;V)} \leq C \left\| \partial_t f_T \right\|_{L^2(0,T;V^*)}
\]
for each $ T>0 $. Here we recall that the constant $C$ is independent of $T$. In addition, if $\partial_t f \in L^2(0,\infty;V^*)$, one can then deduce from the arbitrariness of $T$ that $\partial_t z \in L^2(0,\infty;V)$ and
\[
\left\| \partial_t z \right\|_{L^2(0,\infty;V)} \leq C \left\| \partial_t f \right\|_{L^2(0,\infty;V^*)} < +\infty.
\]
Finally, the continuous dependence on prescribed data (and uniqueness) for global-in-time strong solutions to the initial-boundary value problem can be proved  immediately thanks to Theorem \ref{thm:main}; indeed, the constant $C$ in \eqref{conti-dep} is independent of $T$. This completes the proof.
\end{proof}

Let us move on to a proof for Theorem \ref{thm:asympt}.

\begin{proof}[Proof of Theorem {\rm \ref{thm:asympt}}]
Let $ z $ be the global-in-time strong solution to \eqref{eq:vi} and \eqref{bcic} (see Theorem \ref{thm:global} for the existence).
Define subsets $ J_1 $, $ J_2 $ and $J $ of $ \R_+=(0,\infty)$ by 
\begin{align*}
J_1 &:= \left\{ s \in \R_+ \colon \dind \left( \partial_t z(s) \right) -\Delta z(s) + \sigma z(s) \ni f(s) \quad \mbox{in } \ L^2(\Omega) \right\},\\
J_2 &:= \left\{ s \in \R_+ \colon \left\| -\Delta z(s) + \sigma z(s) \right\|^2_{L^2(\Omega)} \leq 2 \left( \| \hat{f} \wedge (-\Delta z_0 + \sigma z_0) \|^2_{L^2(\Omega)} + \left\| f \right\|^2_{L^\infty(0,T;L^2(\Omega))} \right) \right\},\\
J &:= J_1 \cap J_2,
\end{align*}
where $ \hat{f} $ is a function given in $ (\mathrm{ii})_{T=\infty} $ of Theorem \ref{thm:asympt}.
We first claim that $\R_+ \setminus J$ is negligible, i.e., $|\R_+ \setminus J| = 0$.
Indeed, clearly $ \R_+ \setminus J_1 $ is negligible.
In order to show that $ \R_+ \setminus J_2 $ is negligible, let $ T>0 $, $ m \in \N $ and set $ \tau = T/m > 0 $.
As in the proof of Theorem \ref{thm:main}, define $ f_k \in L^2(\Omega) $ for $ k=1,\dots,m $ by \eqref{eq:defoffk} and let $ z_k \in X \cap V $ be the strong solution to \eqref{eq:dis}.
Moreover, define $ \ztaubar \in L^2(0,T;X \cap V) $ by \eqref{def:ztaubar}.
From \eqref{eq:ztaubar} and Fatou's lemma, we observe that
\[
\int^T_0 \liminf_{\tau \to 0_+} \left\| \Delta \ztaubar(t) \right\|^2_{L^2(\Omega)} \d t \leq \liminf_{\tau \to 0_+} \left\| \Delta \ztaubar \right\|^2_{L^2(0,T;L^2(\Omega))} \leq C,
\]
which implies
\[
\liminf_{\tau \to 0_+} \left\| \Delta \ztaubar(t) \right\|_{L^2(\Omega)} < +\infty \quad \mbox{ for a.e.~} t \in (0,T).
\]
Thus for a.e.~$ t \in (0,T) $, one can take a (not relabeled) subsequence of $ (\tau) $ (which may depend on $ t $) such that
\[
\ztaubar(t) \wto z(t) \quad \mbox{ weakly in } \ X.
\]
Therefore, by virtue of the weak lower-semicontinuity of $ \left\| \, \cdot \, \right\|^2_{L^2(\Omega)} $ in $L^2(\Omega)$, it follows from \eqref{eq:pointwise} that
\begin{align*}
\left\| -\Delta z(t) + \sigma z(t) \right\|^2_{L^2(\Omega)} &\leq \liminf_{\tau \to 0_+} \left\| -\Delta \ztaubar(t) + \sigma \ztaubar(t) \right\|^2_{L^2(\Omega)}\\
&\leq 2 \liminf_{\tau \to 0_+} \left( \| \hat{f} \wedge \left( -\Delta z_0 + \sigma z_0 \right) \|^2_{L^2(\Omega)} + \| \ftaubar(t) \|^2_{L^2(\Omega)} \right)\\
&\leq 2 \left( \| \hat{f} \wedge \left( -\Delta z_0 + \sigma z_0 \right) \|^2_{L^2(\Omega)} + \left\| f \right\|^2_{L^\infty(0,T;L^2(\Omega))} \right)
\end{align*}
for a.e.~$ t \in (0,T) $.
Here we used the fact that $ \| \ftaubar(t) \|_{L^2(\Omega)} \leq \| f \|_{L^\infty(0,T;L^2(\Omega))} < +\infty $ for $ t \in (0,T) $ by assumption.
Thus $ (0,T) \setminus J_2 $ is negligible; moreover, from the arbitrariness of $ T>0 $, one can deduce that $ \R_+ \setminus J_2 $ is negligible.
Therefore $ \R_+ \setminus J $ is negligible.

Using \eqref{dtz-bdd} and assumption (v) of Theorem \ref{thm:asympt}, we see that
\[
\sum_{n=0}^\infty \int_{(n,n+1) \cap J} \Big( \left\| \partial_t z(s) \right\|_V^2 + \left\| f(s)-f_\infty \right\|_{V^*}^2 \Big) \, \d s < \infty.
\]
Hence one can take a sequence $(s_n)$ in $J$ such that $n < s_n < n+1$ for $n \in \N$ and
\begin{alignat*}{4}
\partial_t z(s_n) &\to 0 & \quad &\mbox{ strongly in } \ V,\nonumber\\
f(s_n) &\to f_\infty &&\mbox{ strongly in } \ V^*
\end{alignat*}
as $n \to \infty$. Set $\eta := \Delta z - \sigma z + f \in L^2_\loc([0,\infty);L^2(\Omega))$. Testing it by $\partial_t z$ along with the relation $\eta \partial_t z = 0$ a.e.~in $\Omega \times \R_+$, we find that
\begin{align*}
\frac{1}{2} \frac{\d}{\d t} \left\| \nabla z(t) \right\|^2_{L^2(\Omega)} + \frac{\sigma}{2} \frac{\d}{\d t} \left\| z(t) \right\|^2_{L^2(\Omega)} &= \left\langle f(t), \partial_t z(t) \right\rangle_V\\
&= \left\langle f(t)-f_\infty, \, \partial_t z(t) \right\rangle_V + \left\langle f_\infty, \, \partial_t z(t) \right\rangle_V\\
&\leq \left\| f(t)-f_\infty \right\|_{V^*} \left\| \partial_t z(t) \right\|_V + \frac{\d}{\d t} \left\langle f_\infty, z(t) \right\rangle_V
\end{align*}
for a.e.~$ t \in \R_+ $. Thus
\begin{align*}
\mel{
\frac12\left\| \nabla z(t) \right\|^2_{L^2(\Omega)} + \frac\sigma2 \left\| z(t) \right\|^2_{L^2(\Omega)} - \langle f_\infty, z(t) \rangle_V
}\\
&\leq \frac12\left\| \nabla z_0 \right\|^2_{L^2(\Omega)} + \frac\sigma2 \left\| z_0 \right\|^2_{L^2(\Omega)} - \langle f_\infty, z_0 \rangle_V + \int^t_0 \left\| f(s)-f_\infty \right\|_{V^*} \left\| \partial_t z(s) \right\|_V \d s
\end{align*}
for all $t \in \R_+$. Therefore (with the aid of Poincar\'{e}'s inequality when $\sigma=0$) one has
\[
\left\| z(t) \right\|^2_V \leq C \left( \left\| z_0 \right\|^2_V + \left\| f_\infty \right\|^2_{V^*} + \left\| f-f_\infty \right\|_{L^2(0,\infty;V^*)} \left\| \partial_t z \right\|_{L^2(0,\infty;V)} \right)
\]
for all $t \in \R_+$. Hence we derive that $ z \in L^\infty(0,\infty;V)$.
Moreover, since $ s_n \in J_2 $ for all $ n \in \N $, it follows that
\begin{align*}
\left\| - \Delta z(s_n) + \sigma z(s_n) \right\|^2_{L^2(\Omega)} &\leq 2 \left( \| \hat{f} \wedge \left( -\Delta z_0 + \sigma z_0 \right) \|^2_{L^2(\Omega)} + \left\| f \right\|^2_{L^\infty(0,\infty;L^2(\Omega))} \right),
\end{align*}
and thus, the sequence $(-\Delta z(s_n) + \sigma z(s_n))$ turns out to be bounded in $L^2(\Omega)$. Therefore, since $ (z(s_n)) $ is bounded in $ L^2(\Omega) $, so does $(\Delta z(s_n))$ in $L^2(\Omega)$.
Hence there exist a (not relabeled) subsequence of $(n)$ and $z_\infty \in X \cap V$ such that
\begin{alignat}{4}
z(s_n) &\to z_\infty & \quad &\mbox{ strongly in } \ V,\label{eq:zsnst}\\
\Delta z(s_n) &\to \Delta z_\infty &&\mbox{ weakly in } \ L^2(\Omega).\notag
\end{alignat}
Noting that $z(s_n) \leq z_0$ and $-\Delta z(s_n) + \sigma z(s_n) - f(s_n) \leq 0$ a.e.~in $\Omega$ for $n \in \N$, we readily find that
\begin{equation}\label{eq:zinfvi}
z_\infty \leq z_0 \quad \mbox{ and } \quad -\Delta z_\infty + \sigma z_\infty - f_\infty \leq 0 \quad \mbox{ a.e.~in } \ \Omega.
\end{equation}
Set $\eta(t) = \Delta z(t) - \sigma z(t) + f(t)$. Then, by \eqref{eq:zinfvi}, we have
\[
\eta(t) - \Delta \left( z(t) - z_\infty \right) + \sigma \left( z(t) - z_\infty \right) \geq f(t) - f_\infty \quad \mbox{ for a.e.~} t \in \R_+.
\]
Testing it by $\partial_t z \leq 0$ and using the relation $\eta \partial_t z = 0$ a.e.~in $\Omega \times \R_+$, we have
\[
\frac{1}{2} \frac{\d}{\d t} \left( \left\| \nabla (z(t) - z_\infty) \right\|^2_{L^2(\Omega)} + \sigma \left\| z(t) - z_\infty \right\|^2_{L^2(\Omega)} \right) \leq \langle f(t) - f_\infty, \, \partial_t z(t) \rangle_V,
\]
whence it follows that
\begin{equation}\label{eq:ztzinf}
\left\| z(t) - z_\infty \right\|^2_V \leq C \left( \left\| z(s_n) - z_\infty \right\|^2_V + \int^\infty_{s_n} \left\| f(t) - f_\infty \right\|_{V^*} \left\| \partial_t z(t) \right\|_V \d t \right)
\end{equation}
for all $t \geq s_n$. Since $\partial_t z \in L^2(0,\infty;V)$ (see \eqref{dtz-bdd}) and $f - f_\infty \in L^2(0,\infty;V^*)$, we deduce that
\begin{equation}\label{eq:ffpz}
\lim_{n \to \infty} \int^\infty_{s_n} \left\| f(t) - f_\infty \right\|_{V^*} \left\| \partial_t z(t) \right\|_V \d t = 0.
\end{equation}
Thus combining \eqref{eq:zsnst}, \eqref{eq:ztzinf} and \eqref{eq:ffpz}, we conclude that
\[
\left\| z(t) - z_\infty \right\|^2_V \to 0 \quad \mbox{ as } \ t \to \infty.
\]

Finally, we shall show that $z_\infty \in X \cap V$ is a strong solution to \eqref{eq:viinfty} and \eqref{bc} when \eqref{eq:xi} is satisfied. By virtue of Proposition \ref{pro:static}, there exists a unique strong solution $\zeta \in X \cap V$ to \eqref{eq:viinfty} and \eqref{bc}. We claim that $z_\infty = \zeta$. Indeed, $w(t) \equiv \zeta$ turns out to be the strong solution of the initial-boundary problem,
\begin{equation*}
\left\lbrace
\begin{alignedat}{4}
&\dind \left( \partial_t w \right) -\Delta w + \sigma w \ni f_\infty \quad &&\mbox{ in } \ \Omega \times \R_+,\\
&w = 0 && \mbox{ on } \ \ddom \times \R_+,\\
&\partial_\nu w = 0 && \mbox{ on } \ \ndom \times \R_+,\\
&w(\cdot,0) = \zeta \ &&\mbox{ in } \ \Omega.
\end{alignedat}
\right.
\end{equation*}
Using \eqref{eq:xi} and applying the comparison principle (see Proposition \ref{pro:comp}) to $z$ and $w$, we infer that
\[
\zeta(x) = w(x,t) \leq z(x,t) \quad \mbox{ for a.e.~} (x,t) \in \Omega \times \R_+. 
\]
Hence passing to the limit as $t \to +\infty$, we obtain
\[
\zeta \leq z_\infty \quad \mbox{ a.e.~in } \ \Omega.
\]
On the other hand, it is obvious that $z_\infty$ is a solution to \eqref{eq:static} with $ \psi = z_\infty $ and $ g = -\Delta z_\infty + \sigma z_\infty $.
Moreover, $ \zeta $ is also a solution to \eqref{eq:static} with $ \psi = z_0 $ and $ g = f_\infty $.
Therefore, by the comparison principle for \eqref{eq:static} in Proposition \ref{pro:static}, we can derive from \eqref{eq:zinfvi} that
\[
z_\infty \leq \zeta \quad \mbox{ a.e.~in } \ \Omega.
\]
Thus we get $ z_\infty = \zeta $ (cf.~this fact can also be checked as in \cite[\S 5]{AK}), which completes the proof.
\end{proof}

\section{Characterization as a singular limit}

In this section, we shall prove Theorem \ref{thm:conv}. First of all, we note that the well-posedness of the initial-boundary value problem \eqref{eq-ep}, \eqref{bcic-ep} can be verified under the assumptions of Theorem \ref{thm:conv}, as in~\cite[Theorem 1.8]{AK}, where only the case $\sigma = 0$ is explicitly treated.
However, the method of proofs developed there can also be applied to the case $\sigma \neq 0$ with only obvious modifications.

\begin{proof}[Proof of Theorem {\rm \ref{thm:conv}}]
Set $Z_\vep := z_\vep - z$ and $F_\vep := f_\vep - f$. By subtraction of \eqref{eq} from \eqref{eq-ep}, we see that
$$
\vep \partial_t z_\vep + \dind \left( \partial_t z_\vep \right) - \dind \left( \partial_t z \right) - \Delta Z_\vep + \sigma Z_\vep \ni F_\vep \quad \mbox{ a.e.~in } \ \Omega \times (0,T).
$$
Testing this by $\partial_t Z_\vep$ and using the monotonicity of $\dind$, we find that
\begin{align*}
\MoveEqLeft{
\vep \left\| \partial_t z_\vep(t) \right\|_{L^2(\Omega)}^2 - \vep \left( \partial_t z_\vep(t), \partial_t z(t)\right)_{L^2(\Omega)} + \frac 12 \frac \d{\d t} \left( \|\nabla Z_\vep(t)\|_{L^2(\Omega)}^2 + \sigma \|Z_\vep(t)\|_{L^2(\Omega)}^2 \right)
}\\
&\leq \left( F_\vep(t), \partial_t Z_\vep(t) \right)_{L^2(\Omega)}
= \dfrac \d{\d t} \left\langle F_\vep(t), Z_\vep(t) \right\rangle_{V} - \left\langle \partial_t F_\vep(t), Z_\vep(t) \right\rangle_V
\end{align*}
(see Proposition \ref{pro:leibniz} to verify the last equality), which along with Young's and Schwartz's inequalities yields
\begin{align*}
\MoveEqLeft{
\frac \vep 2 \left\| \partial_t z_\vep(t) \right\|_{L^2(\Omega)}^2 + \frac 12 \frac \d{\d t} \left( \left\| \nabla Z_\vep(t) \right\|_{L^2(\Omega)}^2 + \sigma \left\| Z_\vep(t) \right\|_{L^2(\Omega)}^2 \right)
}\\
&\leq \frac \vep 2 \left\| \partial_t z(t) \right\|_{L^2(\Omega)}^2 + \dfrac \d{\d t} \left\langle F_\vep(t), Z_\vep(t) \right\rangle_V + \left\| \partial_t F_\vep(t) \right\|_{V^*} \left\| Z_\vep(t) \right\|_V
\end{align*}
for a.e.~$t \in (0,T)$. Integrating both sides in time, we infer that
\begin{align*}
\MoveEqLeft{
\frac \vep 2 \int^t_0 \left\| \partial_t z_\vep(s) \right\|_{L^2(\Omega)}^2 \d s + \frac 12 \left( \|\nabla Z_\vep(t)\|_{L^2(\Omega)}^2 + \sigma \left\| Z_\vep(t) \right\|_{L^2(\Omega)}^2 \right)
}\\
&\leq \frac 12 \left( \left\| \nabla Z_\vep(0) \right\|_{L^2(\Omega)}^2 + \sigma \left\| Z_\vep(0) \right\|_{L^2(\Omega)}^2 \right) + \frac \vep 2 \int^T_0 \left\| \partial_t z(s) \right\|_{L^2(\Omega)}^2 \d s \\
&\quad + \left\langle F_\vep(t), Z_\vep(t) \right\rangle_V - \left\langle F_\vep(0), Z_\vep(0) \right\rangle_V + \int^t_0 \left\| \partial_t F_\vep(s) \right\|_{V^*} \left\| Z_\vep(s) \right\|_V \d s
\end{align*}
for any $t \in [0,T]$. Thus it follows (see~\cite[Lemma A.5]{HBF}) that
\begin{align*}
\MoveEqLeft{
\sup_{t \in [0,T]} \left\| Z_\vep(t) \right\|_{V}
}\\
&\leq C \Big( \left\| Z_\vep(0) \right\|_V + \sup_{t \in [0,T]} \left\| F_\vep(t) \right\|_{V^*} + \sqrt{\vep} \left\| \partial_t z \right\|_{L^2(0,T;L^2(\Omega))} + \left\| \partial_t F_\vep \right\|_{L^1(0,T;V^*)} \Big)
\end{align*}
for some constant $C>0$ independent of $T$. This completes the proof.
\end{proof}

\appendix

\section{Discretization and interpolation}\label{sec:intpl}

In this appendix, we shall prove some estimates for time-discrete approximations of $f \in L^2(0,T;L^2(\Omega)) \cap W^{1,2}(0,T;V^*)$. Let $m \in \N$, $\tau = T/m > 0$ and $t_k = k\tau$ for $k = 0,1,\ldots,m$. Moreover, let $\{f_k\}_{k=0,1,\ldots,m}$ be defined by \eqref{eq:defoffk} and \eqref{eq:defoff0}. Then we note that $f_k \in L^2(\Omega)$ for $ k=1,\dots,m $ and $f_0 \in V^*$. Furthermore, denote by $\ftau$ and $\ftaubar$ the piecewise linear and constant interpolants of $\{f_k\}_{k=0,1,\ldots,m}$ defined by \eqref{def:ftau} and \eqref{def:ftaubar}, respectively (in addition, we set $\ftau(\cdot,0)=\ftaubar(\cdot,0):=f_0$).

\begin{prop}\label{pro:inter}
Let $T \in (0,\infty)$ and let $f \in W^{1,2}(0,T;V^*)$. It then holds that
\[
\left\| \partial_t \ftau \right\|^2_{L^2(0,T;V^*)} = \tau \sum^m_{k=1} \left\| \frac{f_k-f_{k-1}}{\tau} \right\|^2_{V^*} \leq 4 \left\| \partial_t f \right\|^2_{L^2(0,T;V^*)}
\]
for any $m \in \N$ {\rm (}or $\tau>0${\rm )}.
\end{prop}

\begin{proof}
Let $v \in V$ be fixed. For any $k = 2,\ldots,m$, we see that
\begin{align*}
\left\langle \frac{f_k-f_{k-1}}{\tau}, v \right\rangle_V 
&= \frac{1}{\tau^2} \left( \int^{t_k}_{t_{k-1}} \left\langle f(s), v \right\rangle_V \d x - \int^{t_{k-1}}_{t_{k-2}} \left\langle f(s), v \right\rangle_V \d s \right)\\
&= \frac{1}{\tau^2} \int^{t_k}_{t_{k-1}} \left\langle f(s) - f(s-\tau), \, v \right\rangle_V \d s.
\end{align*}
By virtue of the fundamental theorem of calculus, it follows that
\begin{align*}
\frac{1}{\tau^2} \int^{t_k}_{t_{k-1}} \left\langle f(s) - f(s-\tau), \, v \right\rangle_V \d s 
&= \frac{1}{\tau^2} \int^{t_k}_{t_{k-1}} \int^s_{s-\tau} \left\langle \partial_t f(r), v \right\rangle_V \d r \,\d s\\
&\leq \left( \frac{1}{\tau^2} \int^{t_k}_{t_{k-1}} \int^s_{s-\tau} \left\| \partial_t f(r) \right\|_{V^*} \d r \,\d s \right) \|v\|_V\\
&\leq \left( \frac{1}{\tau^2} \int^{t_k}_{t_{k-1}} \int^{t_k}_{t_{k-2}} \left\| \partial_t f(r) \right\|_{V^*} \d r\, \d s \right) \|v\|_V\\
&= \left( \frac{1}{\tau} \int^{t_k}_{t_{k-2}} \left\| \partial_t f(r) \right\|_{V^*} \d r \right) \|v\|_V.
\end{align*}
Therefore one can deduce that
\begin{equation}\label{est:sabun}
\left\| \frac{f_k-f_{k-1}}{\tau} \right\|_{V^*} \leq \frac{1}{\tau} \int^{t_k}_{t_{k-2}} \left\| \partial_t f(r) \right\|_{V^*} \d r
\end{equation}
for $k=2,\dots,m$. Moreover, noting that
\begin{align*}
\left\langle \frac{f_1-f_0}{\tau}, v \right\rangle_V &= \frac{1}{\tau} \left( \frac{1}{\tau} \int^{t_1}_{t_0} \left\langle f(s), v \right\rangle_V \d s - \left\langle f(0), v \right\rangle_V \right)\\
&= \frac{1}{\tau^2} \int^\tau_0 \left\langle f(s)-f(0), v \right\rangle_V \d s\\
&= \frac{1}{\tau^2} \int^\tau_0 \int^s_0 \left\langle \partial_t f(r), v \right\rangle_V \d r \, \d s\\
&\leq \left( \frac{1}{\tau^2} \int^\tau_0 \int^s_0 \left\| \partial_t f(r) \right\|_{V^*} \d r \, \d s \right) \|v\|_V\\
&\leq \left( \frac{1}{\tau} \int^\tau_0 \left\| \partial_t f(r) \right\|_{V^*} \d r  \right) \|v\|_V,
\end{align*}
we obtain
\begin{equation}\label{est:sabun1}
\left\| \frac{f_1-f_0}{\tau} \right\|_{V^*} \leq \frac{1}{\tau} \int^\tau_0 \left\| \partial_t f(r) \right\|_{V^*} \d r.
\end{equation}
Therefore we derive from \eqref{est:sabun} and \eqref{est:sabun1} that
\begin{align*}
\tau \sum^m_{k=1} \left\| \frac{f_k-f_{k-1}}{\tau} \right\|^2_{V^*} &\leq \tau \left( \frac{1}{\tau} \int^\tau_0 \left\| \partial_t f(r) \right\|_{V^*} \d r \right)^2 + \tau \sum^m_{k=2} \left( \frac{1}{\tau} \int^{t_k}_{t_{k-2}} \left\| \partial_t f(r) \right\|_{V^*} \d r \right)^2\\
&\leq \int^\tau_0 \left\| \partial_t f(r) \right\|^2_{V^*} \d r + 2 \sum^m_{k=2} \int^{t_k}_{t_{k-2}} \left\| \partial_t f(r) \right\|^2_{V^*} \d r\\
&\leq 4 \left\| \partial_t f \right\|^2_{L^2(0,T;V^*)},
\end{align*}
which completes the proof.
\end{proof}

\begin{prop}\label{pro:f}
Let $f \in L^2(0,T;L^2(\Omega))$. It then holds that
\begin{alignat*}{4}
\ftaubar &\to f & \quad &\mbox{ strongly in } \ L^2(0,T;L^2(\Omega))
\end{alignat*}
as $\tau \to 0_+$.
\end{prop}

\begin{proof}
For any $\vep > 0$, one can take $g_\vep \in C_c(0,T;L^2(\Omega))$ such that $\|f - g_\vep\|_{L^2(0,T;L^2(\Omega))} < \vep$ (see, e.g.,~\cite{CaHa}). From the uniform continuity of $g_\vep$ on $[0,T]$ in the strong topology of $L^2(\Omega)$, it follows that
\begin{align*}
\big\| \overline{(g_\vep)}_\tau - g_\vep \big\|^2_{L^2(0,T;L^2(\Omega))} 
&= \sum^m_{k=1} \int^{t_k}_{t_{k-1}} \left\| \frac{1}{\tau} \int^{t_k}_{t_{k-1}} g_\vep(s) \, \d s  - g_\vep(t) \right\|_{L^2(\Omega)}^2 \d t\\
&\leq \sum^m_{k=1} \int^{t_k}_{t_{k-1}} \left( \frac{1}{\tau} \int^{t_k}_{t_{k-1}} \|g_\vep(s)-g_\vep(t)\|_{L^2(\Omega)}  \, \d s  \right)^2 \d t\\
&\leq T \sup_{\substack{t,s \in (0,T) \\ |t-s|\leq\tau}} \left\| g_\vep(s)-g_\vep(t) \right\|^2_{L^2(\Omega)}  \to 0
\end{align*}
as $\tau \to 0_+$. We also find that
\[
\big\| \overline{(g_\vep)}_\tau - \overline{f}_\tau \big\|_{L^2(0,T;L^2(\Omega))}^2 = \big\| \overline{(g_\vep -f)}_\tau \big\|_{L^2(0,T;L^2(\Omega))}^2 \leq \left\| g_\vep - f \right\|_{L^2(0,T;L^2(\Omega))}^2 \to 0
\]
as $\vep \to 0_+$. Combining all these facts, we obtain the desired fact.
\end{proof}

\section{A Leibniz rule}\label{sec:leibniz}

In this section, we shall provide a slightly generalized Leibniz rule for vector-valued functions.
Let $ H $ be a Hilbert space equipped with an inner product $ (\cdot, \cdot)_H $ whose dual space $ H^* $ is identified with itself and let $ V $ be a Banach space continuously embedded into $ H $ such that
\[
V \hookrightarrow H \simeq H^* \hookrightarrow V^*
\]
and
\[
\left\langle u, v \right\rangle_V = \left( u, v \right)_H \quad \mbox{ for } \ u \in H, \ v \in V.
\]

If $ z \in W^{1,2}(0,T;V) $ and $ f \in W^{1,2}(0,T;V^*) $, then a standard Leibniz rule clearly holds, i.e., the function $ t \mapsto \langle f(t), z(t) \rangle_V $ is absolutely continuous on $ [0,T] $ and
\begin{equation}\label{eq:leibniz}
\frac{\d}{\d t} \left\langle f(t), z(t) \right\rangle_V = \left\langle \partial_t f(t), z(t) \right\rangle_V + \left\langle f(t), \partial_t z(t) \right\rangle_V \quad \mbox{ for a.e.~} t \in (0,T).
\end{equation}
As for the case where $ \partial_t z(t) $ may not lie on $ V $, our result reads,

\begin{prop}\label{pro:leibniz}
Let $ z \in W^{1,2}(0,T;H) \cap L^2(0,T;V) $ and $ f \in L^2(0,T;H) \cap W^{1,2}(0,T;V^*) $.
Then the function $ t \mapsto \langle f(t),z(t) \rangle_V $ is well defined and absolutely continuous on $ [0,T] $ and
\[
\frac{\d}{\d t} \left\langle f(t),z(t) \right\rangle_V = \left\langle \partial_t f(t), z(t) \right\rangle_V + \left( f(t), \partial_t z(t) \right)_H \quad \mbox{ for a.e.~} t \in (0,T).
\]
\end{prop}

\begin{proof}
Let $ \bar{z}=\bar{z}(x,t) \in W^{1,2}(\R;H) \cap L^2(\R;V) $ be an extension of $ z $ onto $ \R $, that is, $ \bar{z}|_{[0,T]} = z $.
Let $ (\rho_n) \subset C^\infty_c (\R) $ be a sequence of mollifiers and set $ z_n := \bar{z} \ast \rho_n $, i.e., $ z_n(x,t) = \int_{\R} \bar{z}(x,s) \rho_n(t-s) \, \d s $ for $ (x,t) \in \Omega \times \R $ and $ n \in \N $.
Then $ (z_n) \subset C^\infty(\R;V) $ and it satisfies that
\begin{equation}\label{eq:density1}
z_n \to \bar{z} \quad \mbox{ in } \ W^{1,2}(\R;H) \cap L^2(\R;V).
\end{equation}
Moreover, there exists a (not relabeled) subsequence of $ (n) $ such that
\begin{equation}\label{eq:density2}
z_n(t) \to \bar{z}(t) \quad \mbox{ in } \ V \quad \mbox{ for a.e.~} t \in \R.
\end{equation}
The standard Leibniz rule \eqref{eq:leibniz} implies that
\begin{align*}
\frac{\d}{\d t} \left\langle f(t), z_n(t) \right\rangle_V &= \left\langle \partial_t f(t), z_n(t) \right\rangle_V + \left\langle f(t), \partial_t z_n(t) \right\rangle_V\\
&= \left\langle \partial_t f(t), z_n(t) \right\rangle_V + \left( f(t), \partial_t z_n(t) \right)_H \quad \mbox{ for a.e.~} t \in (0,T),
\end{align*}
which yields
\[
\left\langle f(t), z_n(t) \right\rangle_V - \left\langle f(s), z_n(s) \right\rangle_V = \int^t_s \left\langle \partial_t f(r), z_n(r) \right\rangle_V \d r + \int^t_s \left( f(r), \partial_t z_n(r) \right)_H \d r
\]
for all $ s,t \in [0,T] $.
Moreover, by virtue of \eqref{eq:density1} and \eqref{eq:density2}, one can deduce that
\begin{align*}
\left\langle f(t), z_n(t) \right\rangle_V &\to \left\langle f(t), z(t) \right\rangle_V \quad \mbox{ for a.e.~} t \in (0,T)
\end{align*}
and
\begin{alignat*}{4}
\int^t_s \left\langle \partial_t f(r), z_n(r) \right\rangle_V \d r &\to \int^t_s \left\langle \partial_t f(r), z(r) \right\rangle_V \d r &\quad& \mbox{ for all } \ s,t \in [0,T],\\
\int^t_s \left( f(r), \partial_t z_n(r) \right)_H \d r &\to \int^t_s \left( f(r), \partial_t z(r) \right)_H \d r &&\mbox{ for all } \ s,t \in [0,T].
\end{alignat*}
Thus we obtain
\begin{equation}\label{eq:leibniz-ae}
\left\langle f(t), z(t) \right\rangle_V - \left\langle f(s), z(s) \right\rangle_V = \int^t_s \left\langle \partial_t f(r), z(r) \right\rangle_V \d r + \int^t_s \left( f(r), \partial_t z(r) \right)_H \d r
\end{equation}
for a.e.~$ s,t \in (0,T) $.
Therefore the function $ t \mapsto \langle f(t), z(t) \rangle_V $ can be extended onto $ [0,T] $; indeed, the value of the function at arbitrary $ t \in [0,T] $ is uniquely determined as the limit of some Cauchy sequence of the form $ (\langle f(t_n), z(t_n) \rangle_V) $, where $ t_n \to t $ and \eqref{eq:leibniz-ae} holds with $ t=t_n $ and $ s=t_m $ for any $ m,n \in \N $ large enough.
Moreover, we infer that it is absolutely continuous on $ [0,T] $ and
\[
\frac{\d}{\d t} \left\langle f(t),z(t) \right\rangle_V = \left\langle \partial_t f(t), z(t) \right\rangle_V + \left( f(t), \partial_t z(t) \right)_H \quad \mbox{ for a.e.~} t \in (0,T),
\]
which completes the proof.
\end{proof}

\section*{Acknowledgments}

The authors sincerely thank the anonymous referee for his/her careful reading, valuable comments and helpful suggestions. The authors wish to express their gratitude to Professor Masato Kimura for several helpful comments and many stimulating conversations. This work was supported by the Research Institute for Mathematical Sciences, an International Joint Usage/Research Center located in Kyoto University.


\end{document}